\documentclass[11pt]{amsart}

\usepackage{epigamath}

\usepackage[english]{babel}

\numberwithin{equation}{section}

\usepackage{enumitem}

\usepackage[utf8]{inputenc} 
\usepackage[T1]{fontenc}   
\usepackage{comment} 

\usepackage[french]{algorithm2e}
\usepackage{enumitem}
\usepackage{mathtools, bm}
\usepackage{amssymb, bm}
\usepackage{amsmath}
\usepackage{amsthm}
\usepackage{hyperref}
\usepackage{stmaryrd}
\usepackage{ mathrsfs }
\usepackage[all]{xy}
\usepackage{scalerel,stackengine}

\usepackage{cleveref}
\usepackage{tabularx}
\usepackage{epigraph}

\usepackage{tkz-graph}
\usetikzlibrary{positioning}
\GraphInit[vstyle = normal]
\tikzset{
  LabelStyle/.style = { rectangle, rounded corners, draw,
                        minimum width = 2em, fill = white, font = \bfseries },
  VertexStyle/.append style = { inner sep=5pt,
                                font = \Large\bfseries},
  EdgeStyle/.append style = {bend left} }

\newtheorem{thm}{Theorem}[section]
\newtheorem{lem}[thm]{Lemma}
\newtheorem{cor}[thm]{Corollary}
\newtheorem{prop}[thm]{Proposition}

\theoremstyle{definition}
\newtheorem{defi}[thm]{Definition}

\theoremstyle{remark}
\newtheorem{rem}[thm]{Remark}
\newtheorem{ex}[thm]{Example}

\newcommand{\ra}{\rightarrow}

\stackMath
\newcommand\reallywidehat[1]{%
	\savestack{\tmpbox}{\stretchto{%
			\scaleto{%
				\scalerel*[\widthof{\ensuremath{#1}}]{\kern-.6pt\bigwedge\kern-.6pt}%
				{\rule[-\textheight/2]{1ex}{\textheight}}
			}{\textheight}%
		}{0.5ex}}%
	\stackon[1pt]{#1}{\tmpbox}%
}

\renewcommand{\P}{\mathbb{P}}
\newcommand{\C}{\mathbb{C}}
\newcommand{\on}[1]{\operatorname{#1}}

\SetKwInOut{Input}{Entrée}
\SetKwInOut{Output}{Sortie}
\SetKwComment{Comment}{}{}

\SetCommentSty{mycmtsty}
\renewcommand{\O}{\mathcal{O}}

\newcommand{\Z}{\mathbb{Z}}
\newcommand{\N}{\mathbb{N}}
\newcommand{\A}{\mathbb{A}}
\newcommand{\m}{\mathfrak{m}}

\newcommand{\p}{\mathfrak{p}}

\DeclareMathOperator{\spec}{Spec}
\DeclareMathOperator{\pic}{Pic}

\DeclareMathOperator{\sing}{Sing}

\DeclareMathOperator{\Frac}{Frac}

\DeclareMathOperator{\Hom}{Hom}

\DeclareMathOperator{\set}{Set}

\DeclareMathOperator{\ab}{Ab}
\DeclareMathOperator{\sch}{Sch}

\DeclareMathOperator{\aut}{Aut}
\DeclareMathOperator{\gal}{Gal}

\renewcommand{\on}[1]{\operatorname{#1}}

\newcommand{\sm}[1]{{{#1}_{\operatorname{sm}}}}
\newcommand{\schsm}[1]{{(\operatorname{Sch}/{#1})_{\operatorname{sm}}}}

\EpigaVolumeYear{6}{2022} \EpigaArticleNr{17} \ReceivedOn{April 9,
2021}
\InFinalFormOn{May 11, 2022}
\AcceptedOn{June 1, 2022}

\title{N\'eron models of Jacobians over bases \\ of arbitrary dimension}
\titlemark{N\'eron models of Jacobians over bases of arbitrary dimension}

\author{Thibault Poiret}
\address{Mathematisch Instituut Leiden, Niels Bohrweg 1, 2333 CA Leiden, The Netherlands}
\email{thibault.poiret5@gmail.com} 

\authormark{Thibault Poiret}

\AbstractInEnglish{We work with a smooth relative curve $X_U/U$ with nodal reduction over an excellent and locally factorial scheme $S$. We show that blowing up a nodal model of $X_U$ in the ideal sheaf of a section yields a new nodal model and describe how these models relate to each other. We construct a N\'eron model for the Jacobian of $X_U$ and describe it locally on $S$ as a quotient of the Picard space of a well-chosen nodal model. We provide a combinatorial criterion for the N\'eron model to be separated.}



\acknowledgement{This work was funded jointly by a Contrat Doctoral Sp\'ecifique pour Normaliens and by Universiteit Leiden. It was carried out at
Universit\'e de Bordeaux and Universiteit Leiden.} 

\begin{document}



\maketitle

\begin{prelims}

\DisplayAbstractInEnglish

\bigskip

\DisplayKeyWords

\medskip

\DisplayMSCclass

\end{prelims}


\newpage

\setcounter{tocdepth}{1}

\tableofcontents


\section{Introduction}

\subsection{N\'eron models}
Given a scheme $S$ and a dense open $U\subset S$, proper and smooth schemes over $U$ often have no proper and smooth model over $S$. Even so, they may still have a canonical smooth $S$-model, the \emph{N\'eron model}, first introduced in~\cite{neron_article}. The N\'eron model of $X_U/U$ is defined as a smooth $S$-model satisfying the \emph{N\'eron mapping property}: for every smooth $S$-scheme $T$ and every $U$-morphism $\phi_U\colon T_U\rightarrow X_U$, there exists a unique morphism $\phi\colon T\ra N$ extending $\phi_U$. N\'eron models are unique up to a unique isomorphism and inherit a group structure from $X_U$ when it has one.

N\'eron proved in the original article~\cite{neron_article} that abelian varieties over a dense open of a Dedekind scheme always have N\'eron models. Recently, people have taken interest in constructing N\'eron models in different settings. It was proved by Qing Liu and Jilong Tong in~\cite{LiuTong} that smooth and proper curves of positive genus over a dense open of a Dedekind scheme always have N\'eron models. This does not apply to genus $0$: if $S$ is the spectrum of a discrete valuation ring with fraction field $K$, then $\mathbb P^1_K$ does not have a N\'eron model. Indeed, the N\'eron model, if it existed, would be the smooth and proper model $\mathbb P^1_S$. But $\mathbb P^1_S$   does not have the N\'eron mapping property since many automorphisms of $\mathbb P^1_K$ do not extend to automorphisms of $\mathbb P^1_S$ (\textit{e.g.}\ multiplication by a uniformizer).

Among the concrete applications of the theory of N\'eron models, we can cite the semi-stable reduction theorem (an abelian variety over the fraction field of a discrete valuation ring acquires semi-abelian reduction after a finite extension of the base field), the N\'eron--Ogg--Shafarevich criterion for good reduction of abelian varieties, the computation of canonical heights on Jacobians, as well as the linear and quadratic Chabauty methods to determine whether or not a list of rational points on a curve is exhaustive. For a geometric description of the quadratic Chabauty method, see~\cite{QuadChabauty}. Parallels can also be drawn to some problems in which N\'eron models do not explicitly intervene, such as extending the double ramification cycle on the moduli stack of smooth curves to the whole moduli stack of stable curves as in~\cite{HolmesExtendingDRC}. Here, one is interested in models in which one given section extends, instead of all sections simultaneously, but the two problems are closely related.

\subsection{Models of Jacobians}

Some constructions have already been made relating to N\'eron models of Jacobians over higher-dimensional bases. When $S$ is a regular scheme of arbitrary dimension and $X/S$ is a nodal curve\footnote{A nodal curve is a proper, flat, finitely presented morphism with geometric fibres of pure dimension $1$ and at worst ordinary double point singularities.} smooth over $U$, David Holmes exhibited in~\cite{Holmes} a combinatorial criterion on $X/S$ called \emph{alignment}, necessary for the Jacobian of $X_U$ to have a separated N\'eron model and sufficient when $X$ is regular. In~\cite{GiulioToricAdd}, Giulio Orecchia introduces the \emph{toric-additivity} criterion. Consider an abelian scheme $A/U$ with semi-abelian reduction $\mathcal{A}/S$, where $S$ is a regular base and $U$ the complement in $S$ of a strict normal crossings divisor. Toric-additivity is a condition on the Tate module of $A$. When $A$ is the generic Jacobian of an $S$-curve with a nodal model, toric-additivity is sufficient for a separated N\'eron model of $A$ to exist. It is also necessary up to some restrictions on the base characteristic. For general abelian varieties, it is proven in~\cite{GiulioMonodromyCriterion} that toric-additivity is still sufficient when $S$ has equicharacteristic $0$, and a partial converse holds; \textit{i.e.}\ the existence of a separated N\'eron model implies a weaker version of toric-additivity. When $S$ is a toroidal variety and $X/S$ a nodal curve, smooth over the complement $U$ of the boundary divisor, a construction of the N\'eron model of the Jacobian of $X_U$ is given in~\cite{HMOPModelsJacobians}, together with a moduli interpretation for it.

Let $g\geq 3$ be an integer. In~\cite{Caporaso2008Neron-models-an}, Lucia Caporaso constructs a "balanced Picard stack" $\mathcal P_g^d$, naturally mapping to the moduli stack $\overline{\mathcal M}_g$ of stable curves of genus $g$. This stack acts as a universal N\'eron model of the degree $d$ Jacobian for test curves corresponding to regular stable curves, \textit{i.e.}\ if $T \to \overline{\mathcal M}_g$ is a morphism from a trait\footnote{A trait is the spectrum of a discrete valuation ring.} such that the corresponding stable curve $X/T$ is regular, then $\mathcal P_g^d\times_{\overline{\mathcal M}_g} T$ is canonically isomorphic to the N\'eron model of the degree $d$ Jacobian of the generic fibre of $X$. The balanced Picard stack does not admit a group structure compatible with that of the Jacobian. In~\cite{HolmesUniversalJacobian}, Holmes exhibits an algebraic space $\widetilde{\mathcal M}_g$ over $\overline{\mathcal M}_g$ which is regular, in which $\mathcal M_g$ is dense, over which the universal Jacobian has a separated N\'eron model, and which is universal with respect to these properties.

\subsection{Notation}\label{notations}

We will adopt the following conventions:
\begin{itemize}
	\item If $f \colon X \to S$ is a morphism of algebraic spaces locally of finite type, we call \emph{smooth locus} of $f$, and denote by $(X/S)^{\mathrm{smooth}}$ (or $X^{\mathrm{smooth}}$ if there is no ambiguity), the open subspace of $X$ at which $f$ is smooth. Likewise, the \emph{\'etale locus} $(X/S)^{\mathrm{\acute{e}tale}}$ (or just $X^{\mathrm{\acute{e}tale}}$) of $f$ is the open subspace of $X$ at which $f$ is \'etale.
	\item If $f \colon X \to S$ is a morphism of schemes which is locally of finite presentation and has geometric fibres of pure dimension $1$, we call \emph{singular locus} of $f$, and denote by $\sing(X/S)$, the closed subscheme of $X$ cut out by the first Fitting ideal of the sheaf of relative $1$-forms of $X/S$. The set-theoretical complement of $\sing(X/S)$ in $X$ is precisely $(X/S)^{\mathrm{smooth}}$.
	\item Unless specified otherwise, if $A$ is a local ring, we write $\mathfrak{m}_A$ for its maximal ideal, $k_A$ for the its residue field and $\widehat{A}$ for its $\mathfrak{m}_A$-adic completion.
	\item When $M$ is a monoid, or sheaf of monoids, we write $\overline M$ for the quotient of $M$ by its units.
\end{itemize}

\subsection{Structure of the paper, main results}

In this article, we work with a nodal curve $X/S$, smooth over a dense open $U\subset S$, where $S$ is an excellent scheme satisfying certain conditions of local factoriality. We are interested in constructing a N\'eron model for the Jacobian of $X_U$.

In Section \ref{Generalities_about_NMs}, we start with a general discussion about the base change properties of N\'eron models, and we show the following result. 

\begin{cor}[\textit{cf.} Corollary~\ref{corollary:quotient_of_group_space_by_E^et_has_uniqueness_in_NMP}]
	Let $S$ be a scheme, $U\subset S$ a scheme-theoretically dense open subscheme, $N_U \to U$ a smooth, separated $U$-group algebraic space and $f\colon N \to S$ a smooth $S$-group model of $N_U$. Denote by $E$ the scheme-theoretic closure of the unit section in $N$ and by $E^{\mathrm{\acute{e}tale}}$  the \'etale locus of $E/S$. Then, for any smooth $S$-scheme $Y$, the sequence of abelian groups
	\[
	0 \to \Hom(Y,E^{\mathrm{\acute{e}tale}}) \to \Hom(Y,N) \to \Hom(Y_U,N_U)
	\]
	is exact. In particular, the quotient space $N/E^{\mathrm{\acute{e}tale}}$ is a smooth $S$-group model of $N_U$ with uniqueness in the N\'eron mapping property.
\end{cor}

In Section \ref{section_intro1}, we present some generalities about nodal curves, their local structure and their dual graphs.

In Section \ref{section_intro2}, we are interested in \emph{smooth-factorial} schemes, \textit{i.e.}\ those schemes $S$ such that any smooth $S$-scheme is locally factorial. We give conditions under which a prime divisor in a smooth-factorial scheme $S$ remains prime in $Y$ for various kinds of morphisms $Y \to S$.

In Section \ref{section_intro3}, we work with a section $\sigma \colon S \to X$. We introduce a combinatorial invariant, the \emph{type} of $\sigma$ at a point $s\in S$. We discuss some properties of this invariant, and we show that there are \'etale quasisections of every possible type through any singular point of $X/S$.

In Section \ref{section1}, we study blow-ups $X' \to X$ in the ideal sheaves of $S$-sections. We show that $X'/S$ is a nodal curve and that, locally on $S$, it is characterized by the type of the section. We compute the smoothing parameters of the nodes of $X'/S$ in terms of those of $X/S$. We show that, \'etale-locally on $S$, we can always obtain a model of $X_U$ satisfying strong conditions of local factoriality by repeatedly blowing up $X$ in $S$-sections. This can be seen as a higher-dimensional variant of the smoothening process of~\cite{NeronModels}. The reader familiar with logarithmic geometry can establish a parallel between these blow-ups and logarithmic modifications of a log curve inducing a given subdivision of its tropicalization (although one should be careful with this analogy; see Remark \ref{remark:strict_log_jac_vs_pic_aggregate}).

In Section \ref{sec6}, we construct the N\'eron model of the Jacobian. We describe how blowing up a nodal curve in a $S$-section affects its relative Picard scheme, giving us a "bigger" model of the Jacobian (\textit{cf.} Lemma \ref{proposition:refinements_induce_open_immersions_of_Pics}). Then, we show that one obtains a N\'eron model by appropriately quotienting a union of such models. The main result is as follows. 

\begin{thm}[\textit{cf.} Theorem \ref{theorem:NMs_Jac}]
Let $S$ be a smooth-factorial $($\textit{e.g.}\ regular$)$ and excellent scheme, and let $U\subset S$ be a dense open subscheme. Let $X_U/U$ be a smooth curve that admits a nodal model over $S$. Then:

\begin{enumerate}

\item The Jacobian $J=\pic^0_{X_U/U}$ of $X_U/U$ admits a N\'eron model $N$ over $S$.

\item For any nodal model $X/S$ of $X_U/U$, the map $\pic^{\on{tot}0}_{X/S}/E^{\mathrm{\acute{e}tale}} \to N$ extending the identity over $U$ is an open immersion, where $E$ is the scheme-theoretic closure of the unit section in $\pic^{\on{tot}0}_{X/S}$.

\item For any \'etale morphism $V \to S$ and nodal $V$-model $X$ of $X_{U\times_S V}$, if $\bar s \to V$ is a geometric point such that the singularities of $X_{\bar s}$ have prime thicknesses, then the canonical map $\pic^{\on{tot}0}_{X/V} \to N$ is surjective on $\spec(\O_{S,\bar s}^{\,\on{\acute{e}t}})$-points.

\end{enumerate}
\end{thm}

When $X/S$ comes from a vertical logarithmic curve over a logarithmically regular base (\textit{e.g.}\ $S$ is regular and $X$ smooth over the complement of a normal crossings divisor, or $S$ is a toroidal variety and $X$ smooth over the complement of the boundary divisor), the N\'eron model exists and has a moduli interpretation by~\cite[Corollary 6.13]{HMOPModelsJacobians}. The difference is that we require $S$ to be smooth-factorial but allow the discriminant locus to be arbitrary.

We give a local description of the N\'eron model in terms of Picard spaces in Remark \ref{remark:N_covered_by_some_Pic/E_after_etale_cover}. Finally, we give a simple way to determine whether or not the N\'eron model is separated. Following the ideas of~\cite{Holmes}, we say $X/S$ is \emph{strictly aligned} when the smoothing parameters of its singularities satisfy a certain combinatorial condition (\textit{cf.} Definition~\ref{c-strict alignment}), and we prove the following result.  

\begin{thm}[\textit{cf.} Theorem \ref{theorem:separatedness_of_NM_jac_iff_strictly_aligned}]\label{Theorem 1.3}
Let $S$ be a regular and excellent scheme, $U\subset S$ a dense open subscheme and $X/S$ a nodal curve, smooth over $U$. Denote by $J$ the Jacobian of $X_U/U$. Then, the $S$-N\'eron model $N$ of $J$ exhibited in Theorem~\ref{theorem:NMs_Jac} is separated if and only if $X/S$ is strictly aligned.
\end{thm}

With the notation of Theorem \ref{Theorem 1.3}, when $X/S$ is strictly aligned, the N\'eron model was already constructed in~\cite{Holmes} (see Proposition 3.6 of \emph{loc.~cit.}) under the additional assumption that $U$ is the complement of a normal crossings divisor in $S$. This additional assumption guarantees the existence of a global nodal model $X'/S$ whose total space is regular, in which case the N\'eron model of the Jacobian is the quotient of $\on{Pic}^{\on{tot} 0}_{X'/S}$ by the closure of its unit section. In our setting, the phenomenon illustrated by Example \ref{example:irred_not_etale_irred} prevents the existence of such an $X'$ in general, but a separated N\'eron model of the Jacobian still exists.

\subsection*{Acknowledgments}
The author would like to thank Qing Liu and the late Bas Edixhoven for their invaluable guidance, David Holmes and Giulio Orecchia for many fruitful discussions regarding N\'eron models of Jacobians, Samouil Molcho for helping get a logarithmic perspective on those models, the late Michel Raynaud for his help in understanding the theory of N\'eron models over Dedekind schemes, Ofer Gabber for his help with the proof of Lemma \ref{les anneaux locaux sont UFD} and Minsheon Shin for providing a reference for Lemma \ref{lemma:factoriality_descends_under_faithfully_flat_ring_ext} in answer to a question on \texttt{math.stackexchange.com}.

\section{Generalities about N\'eron models}\label{Generalities_about_NMs}

\subsection{Definitions}

\begin{defi}
	Let $S$ be a scheme and $U$ a scheme-theoretically dense open subscheme of $S$. Let $Z/U$ be a $U$-algebraic space. An \textit{$S$-model} of $Z$ (or just \textit{model} if there is no ambiguity) is an $S$-algebraic space $X$ together with an isomorphism $X_U=Z$. A \textit{morphism of $S$-models} between two models $X$ and $Y$ of $Z$ is an $S$-morphism $X\ra Y$ that commutes over $U$ with the given isomorphisms $X_U=Z$ and $Y_U=Z$.
\end{defi}

\begin{defi}\label{definition:neron_models}
	Let $S$ be a scheme and $U$ a scheme-theoretically dense open subscheme of $S$. Let $Z/U$ be a smooth $U$-scheme and $N$ an $S$-model of $Z$. We say that $S$ has the \emph{N\'eron mapping property} (resp.\ \emph{existence in the N\'eron mapping property}, \emph{uniqueness in the N\'eron mapping property}) if for each smooth $S$-algebraic space $Y$, the restriction map
	\[
	\Hom_S(Y,N)\ra\Hom_U(Y_U,Z)
	\]
	is bijective (resp.\ surjective, injective).
        If $N$ is $S$-smooth and has the N\'eron mapping property, we say that it is an \emph{$S$-N\'eron model} of $Z$ (or just \emph{N\'eron model} if there is no ambiguity).
\end{defi}

\begin{rem}
	Various authors require N\'eron models to be separated and of finite type over the base. N\'eron models without a quasicompactness condition are sometimes referred to as N\'eron-lft models, where "lft" stands for "locally of finite type". We justify the definition above by observing that our N\'eron models are still unique up to a unique isomorphism by virtue of the universal property: we can always discuss their separatedness or quasicompactness \textit{a posteriori}.
\end{rem}

\begin{rem}
	Let $S$, $U$, $Z$ be as in Definition \ref{definition:neron_models}, and let $N$ be a smooth, separated $S$-model of $Z$. Consider a smooth $S$-algebraic space $Y/S$ and two morphisms $f_1,f_2\colon Y\ra N$ that coincide over $U$. The separatedness of $N/S$ implies that the equalizer of $f_1$ and $f_2$ is a closed subspace of $Y$ containing $Y_U$, and the latter is scheme-theoretically dense in $Y$ by~\cite[th\'eor\`eme 11.10.5]{EGA4.3}. Thus, $N$ automatically has uniqueness in the N\'eron mapping property.
\end{rem}

\begin{rem}\label{remark:NMs_represent_the_restriction_functor_on_small_smooth_site}
	By descent, the definition of N\'eron models is unchanged if we only require the N\'eron mapping property to hold when $Y/S$ is a scheme. Therefore, when we ask if a N\'eron model exists for $X_U/U$, we are asking if the functor $Y \mapsto X_U(Y_U)$ on the small smooth site of $S$ is representable by a smooth algebraic space. Similarly, we would get a looser (but still universal) definition of N\'eron models by asking for them to be smooth algebraic stacks.
\end{rem}

\subsection{Base change and descent properties}

\begin{prop}[The formation of N\'eron models is compatible with smooth base change]\label{changement de base lisse} 
	Consider a smooth morphism $S'\ra S$, a scheme-theoretically dense open $U\subset S$ and a smooth $S$-model $X$ of $X_U$ with uniqueness in the N\'eron mapping property. Then, the base change $X_{S'}$ is a smooth $S'$-model of $X_{U'}$ with uniqueness in the N\'eron mapping property. If $X$ is the N\'eron model of $X_U$, then $X'$ is the N\'eron model of $X_{U'}$.
\end{prop}

\begin{proof}
	First, note that $X_{S'}/S'$ is smooth since $X/S$ is and that $U'$ is scheme-theoretically dense in $S'$ by~\cite[th\'eor\`eme 11.10.5]{EGA4.3}. Thus, we only need to check that $X'/S'$ has uniqueness in the N\'eron mapping property and that it has existence if $X/S$ does.
	
	Let $Y'$ be a smooth $S'$-scheme. A morphism $Y' \to X'$ is uniquely determined by the two projections $Y' \to S'$ and $Y' \to X$. Since $Y' \to S$ is smooth, it follows that $X'$ has uniqueness in the N\'eron mapping property. Now, suppose that $X$ is the N\'eron model of $X_U$, and consider a $U'$-morphism $u': Y'_{U'}\ra X_{U'}$. Composing with the projection: $X_{U'}\ra X_U$, we get a $U$-morphism $Y'_{U'}\ra X_U$, which extends to a unique $S$-morphism $Y'\ra X$ by the N\'eron mapping property since $Y'/S$ is smooth. Then the induced morphism $Y'\ra X'$ extends $u'$.
\end{proof}

\begin{cor}\label{corollaire le NM passe aux limites d'algebres lisses}
	If $S'/S$ is a cofiltered limit of smooth $S$-schemes $($indexed by a cofiltered partially ordered set, \textit{e.g.}\ a localization, a henselization when $S$ is local,\ldots$)$ and $X$ is the $S$-N\'eron model of $X_U$, then $X_{S'}$ is the $S'$-N\'eron model of $X_{U'}$.
\end{cor}

\begin{prop}[N\'eron models descend along smooth covers]\label{proposition descente lisse des NM}
	Let $S$ be a scheme and $U$ a scheme-theoretically dense open of $S$. Let $S'\ra S$ be a smooth surjective morphism and $U'=U\times_S S'$. Let $X_U$ be a smooth $U$-algebraic space, and suppose $X_{U'}$ has a $S'$-N\'eron model $X'$. Then $X'$ comes via base change from an $S$-space $X$, which is the N\'eron model of $X_U$.
\end{prop}

\begin{proof}
Denote by $p_1,p_2$ the two projections $S'':=S'\times_S S'\ra S'$. They are smooth morphisms, so by Proposition~\ref{changement de base lisse}, there is a canonical isomorphism $p_1^*X'=p_2^*X'$ satisfying the cocycle condition. By the effectiveness of fppf descent for algebraic spaces (\textit{cf.} \cite[\href{https://stacks.math.columbia.edu/tag/0ADV}{Tag 0ADV}]{stacks-project}), $X'$ comes via base change from an $S$-algebraic space $X/S$. The morphism $X\ra S$ is smooth since $X'/S'$ is (smoothness is even fpqc local on the base; see~\cite[\href{http://stacks.math.columbia.edu/tag/02VL}{Tag 02VL}]{stacks-project}). Therefore, we only need to prove that $X/S$ has the N\'eron mapping property. Let $Y$ be a smooth $S$-algebraic space and $f_U$ be in $\Hom(Y_U,X_U)$. The corresponding map $Y'_U \to X'_U$ extends uniquely to some $f'\colon Y'\ra X'$. The two pull-backs of $f'$ to $S'':=S'\times_S S'$ coincide over~$U$; hence they coincide (since $X_{S''}$ has the N\'eron mapping property by Proposition \ref{changement de base lisse}). Hence $f'$ comes from a morphism $f \colon Y \to X$, which is the only one extending $f_U$.
\end{proof}

\subsection{Group models with injective restriction map}

In this subsection, we show that any smooth group model of a group algebraic space has a biggest quotient with uniqueness in the N\'eron mapping property, namely the quotient by the \'etale locus over the base of its unit section (\textit{cf.} Corollary \ref{corollary:quotient_of_group_space_by_E^et_has_uniqueness_in_NMP}). Given a scheme $S$, we will write $\sm S$ (resp.\ $\schsm S$) for the small smooth site (resp.\ big smooth site) of $S$.

We write $\ab$ for the category of abelian groups. We will use without further mention the fact that a smooth $S$-group algebraic space is determined up to a unique isomorphism by the corresponding functor $(\sm S)^{\on{op}} \to \ab$ (by combining the Yoneda lemma with a descent argument). Recall from Subsection~\ref{notations} that if $f\colon E \to S$ is a morphism, then $(E/S)^{\mathrm{\acute{e}tale}}$ (or just $E^{\mathrm{\acute{e}tale}}$) denotes the \'etale locus of $f$.

\begin{lem}[\textit{cf.} \protect{\cite[Lemma~5.18]{Holmes}}]\label{lemma:sections_of_E_factor_through_E^et}
	Let $S$ be a scheme, $U\subset S$ a scheme-theoretically dense open and $f\colon E \to S$ an $S$-algebraic space. Suppose that $f$ restricts to an isomorphism over $U$ and that $U\times_S E$ is scheme-theoretically dense in $E$. Then any section of $f$ factors through $E^{\mathrm{\acute{e}tale}}$.
\end{lem}

\begin{proof}
	The claim is \'etale-local on $S$ and $E$, so we may assume that $f$ is a morphism of affine schemes. In particular, $f$ is separated. Then $S \to E$ is a closed immersion through which $U$ factors, hence an isomorphism. \textit{A fortiori}, $f$ is \'etale.
\end{proof}

\begin{cor}\label{corollary:quotient_of_group_space_by_E^et_has_uniqueness_in_NMP}
	Let $S$ be a scheme, $U\subset S$ a scheme-theoretically dense open subscheme, $N_U \to U$ a smooth, separated $U$-group algebraic space and $f\colon N \to S$ a smooth $S$-group model of $N_U$. Denote by $E$ the scheme-theoretic closure of the unit section in $N$. Then, for any smooth $S$-scheme $Y$, the sequence of abelian groups
	\[
	0 \to \Hom(Y,E^{\mathrm{\acute{e}tale}}) \to \Hom(Y,N) \to \Hom(Y_U,N_U)
	\]
	is exact. In particular, the quotient space $N/E^{\mathrm{\acute{e}tale}}$ is a smooth $S$-group model of $N_U$ with uniqueness in the N\'eron mapping property.
\end{cor}

\begin{rem}
	In the setting of Corollary \ref{corollary:quotient_of_group_space_by_E^et_has_uniqueness_in_NMP}, if $N$ has existence in the N\'eron mapping property, it follows that $N/E^{\mathrm{\acute{e}tale}}$ is the N\'eron model of $N_U$.
\end{rem}

\section{Nodal curves and dual graphs}\label{section_intro1}

\subsection{First definitions}

The results of this subsection mostly either are well-known facts about nodal curves or come from~\cite{Holmes}. When the proofs are short enough, we reproduce them for convenience.

\begin{defi}
A \emph{graph} $G$ is an ordered pair of finite sets $(V,E)$, together with a map $f:E\ra (V\times V)/\mathcal{S}_2$. We denote by $V$ the set of vertices of $G$ and by $E$ its set of edges. We think of $f$ as the map sending an edge to its endpoints. We call \emph{loop} any edge in the preimage of the diagonal of $(V\times V)/\mathcal{S}_2$. We will often omit $f$ in the notation and write $G=(V,E)$.

Let $v,v'$ be two vertices of $G$. A \emph{path between $v$ and $v'$ in $G$} is a finite sequence $(e_1,\ldots,e_n)$ of edges such that there are vertices $v_0=v,v_1,\ldots,v_n=v'$ satisfying $f(e_i)=(v_{i-1},v_i)$ for all $1\leq i\leq n$. We denote by $n$ the \emph{length} of the path. A \emph{chain} is a path as above, with positive $n$, where the only repetition allowed in the vertices $(v_i)_{0\leq i\leq n}$ is $v_0=v_n$. A \emph{cycle} is a chain from a vertex to itself. The cycles of length $1$ of $G$ are its loops.

Let $M$ be a monoid. A \emph{labelled graph over $M$} (or \emph{labelled graph} if there is no ambiguity) is the data of a graph $G=(V,E)$ and a map $l \colon E\ra M\backslash\{0\}$, called \emph{edge-labelling}. The image of an edge by this map is called the \emph{label} of that edge.
\end{defi}

\begin{defi}
Let $X$ be an algebraic space. We call \emph{geometric point} of $X$ a morphism $\spec\bar{k}\ra~X$ where the image of $\spec\bar{k}$ is a point with residue field $k$, and $\bar{k}$ is a separable closure of $k$. 
Given two geometric points $s,t$ of $X$, we say that $t$ is an \emph{\'etale generization} of $s$ (or that $s$ is an \emph{\'etale specialization} of $t$) when the image of $t\ra X$ is a generization of the image of $s \to X$. We will often omit the word "\'etale" and just call them specializations and generizations.
\end{defi}

\begin{defi}\label{definition des courbes nodals et points etales}
A \emph{curve} over a separably closed field $k$ is a proper, finitely presented morphism $X\ra\spec k$ with $X$ of pure dimension $1$. It is called \emph{nodal} if it is connected and for every point $x$ of $X$, either $X/k$ is smooth at $x$, or $x$ is an ordinary double point (\textit{i.e.}\ the completed local ring of $X$ at $x$ is isomorphic to $k[[u,v]]/(uv)$).

A \emph{curve} (resp.\ \emph{nodal curve}) over a scheme $S$ is a proper, flat, finitely presented morphism $X\ra S$ whose geometric fibres are curves (resp.\ nodal curves).
\end{defi}

\begin{rem}
By~\cite[Proposition 10.3.7]{Liu}, our definition of nodal curves is unchanged if one defines geometric points with algebraic closures instead of separable closures.
\end{rem}

\begin{defi}
Let $S$ be a scheme, $s$ a point of $S$ and $\bar s$ a geometric point mapping to $s$. We will call \emph{\'etale neighbourhood of $\bar s$ in $S$} the data of an \'etale morphism of schemes $V \to S$, a point $v$ of $V$ and a factorization $\bar s \to v \to s$ of $\bar s \to s$. \'Etale neighbourhoods naturally form a codirected system, and we call \emph{\'etale stalk of $S$ at $s$} the limit of this system. The \'etale stalk of $S$ at $s$ is an affine scheme, and we call \emph{\'etale local ring at $s$}, and denote by $\O_{S,s}^{\,\on{\acute{e}t}}$, its ring of global sections. We will sometimes keep the choice of geometric point $\bar s$ implicit and abusively write $(V,v)$, or even $V$, for an \'etale neighbourhood of $s$ in $S$.
\end{defi}

\begin{rem}
The \'etale local ring of $S$ at $\bar s$ is the strict henselization of the local ring (in the Zariski topology) $\O_{S,s}$ determined by the separable closure $k(s) \to k(\bar s)$.
\end{rem}

\subsection{The local structure}

\begin{prop}\label{structure locale des courbes nodals}
Let $S$ be a locally Noetherian scheme and $X/S$ be a nodal curve. Let $s$ be a geometric point of $S$ and $x$ be a non-smooth point of $X_s$. There exists a unique principal ideal $T=(\Delta)$ of the \'etale local ring $\O^{\,\on{\acute{e}t}}_{S,s}$ such that
\[
\widehat{\O^{\,\on{\acute{e}t}}_{X,x}}\simeq\widehat{\O^{\on{\,\acute{e}t}}_{S,s}}[[u,v]]/(uv-\Delta).
\]
We call $T$ the \emph{thickness} of $x$. It can be seen as an element of the $($multiplicative$)$ monoid $\overline{\O^{\,\on{\acute{e}t}}_{S,s}} = \O^{\,\on{\acute{e}t}}_{S,s}/(\O^{\,\on{\acute{e}t}}_{S,s})^\times$.
\end{prop}

\begin{proof}
	This is~\cite[Proposition 2.5]{Holmes}.
\end{proof}

\begin{rem}
The element $\Delta$ of $\O^{\,\on{\acute{e}t}}_{S,s}$ is a nonzerodivisor if and only if $X/S$ is generically smooth in a neighbourhood of $x$.
\end{rem}

\subsection{The dual graph at a geometric point}

We discuss an important combinatorial object, the \emph{dual graph} of a nodal curve at a geometric point. Throughout the literature, one can find many definitions of dual graphs (or tropicalizations of logarithmic curves), depending on how much information the authors need this object to carry. With a definition slightly heavier than ours, one can construct them functorially in families (see~\cite[Section 3]{HMOPModelsJacobians}).

\begin{defi}\label{definition:dual_graph}
	Let $X/S$ be a nodal curve with $S$ locally Noetherian, and let $s$ be a geometric point of $S$. We define the \emph{dual graph} of $X$ at $s$ as follows:
	\begin{itemize}
	\item Its vertices are the irreducible components of $X_s$. 
	\item It has an edge for every singular point $x$ of $X_s$, whose endpoints are the two (possibly equal) irreducible components containing the two preimages of $x$ in the normalization of $X_s$. 
	\item It has an edge-labelling by the multiplicative monoid $\overline{\O^{\,\on{\acute{e}t}}_{S,s}}$, mapping a singular point to its thickness.
	\end{itemize}
When $S$ is strictly local, we will sometimes refer to the dual graph of $X$ at the closed point as simply "the dual graph of $X$".
\end{defi}

\begin{prop}\label{graphes duaux et changement de base}
	Let $S'\ra S$ be a morphism between locally Noetherian schemes, $X/S$ a nodal curve, $s$ a geometric point of $S$ and $s'$ a geometric point of $S'$ mapping to a generization of $s$. Let $X'$ be the base change of $X$ to $S'$. Let $\Gamma$ $($resp.\ $\Gamma')$ be the dual graph of $X$ at $s$ $($resp.\ of $X'$ at $s')$.
	
	Then, $\Gamma'$ is obtained from $\Gamma$ by contracting the edges whose labels map to $1$ in $M:=\overline{\O}^{\,\on{\acute{e}t}}_{S',s'}$ and replacing the labels of the other edges by their images in $M$.
	
	In particular, if $s'$ has image $s$, then $\Gamma$ and $\Gamma'$ are isomorphic as non-labelled graphs, and the labels of $\,\Gamma'$ are the images in $M$ of those of $\,\Gamma$.
\end{prop}

\begin{proof}
	This is~\cite[Remark 2.12]{Holmes}. We re-prove it here.
	
	We can reduce to $S=\spec R$ and $S'=\spec R'$ affine and strictly local,\footnote{In other words, $S$ and $S'$ are isomorphic to spectra of strictly henselian local rings.} with respective closed points $s$ and~$s'$.
	
	Let $x$ be a singular point of $X$ with image $s$ and $\Delta\in R$ be a lift of its thickness. Then we can choose an isomorphism $\widehat{\O_{X,x}}=\widehat{R}[[u,v]]/(uv-\Delta)$.

	This yields $\widehat{\O_{X,x}}\otimes_R R'=\widehat{R}\otimes_R R'[[u,v]]/(uv-\Delta)$. The ring $\widehat{R}\otimes_R R'$ is local, with completion $\widehat{R'}$ with respect to the maximal ideal: as desired, if $\Delta$ is invertible in $R'$, then $X'$ is smooth above a neighbourhood of $x$, and otherwise, $X'$ has exactly one singular closed point of image $x$, with thickness $\Delta R'$.
\end{proof}

\begin{ex}\label{exemple specialization maps}
With the same notation as above, in the case $S=S'$, we have defined the \emph{specialization maps} of dual graphs: if $s,\xi$ are geometric points of $S$ with $s$ specializing $\xi$, we have a canonical map from the dual graph at $s$ to the dual graph at $\xi$, contracting an edge if and only if its label becomes the trivial ideal in~$\O_{S,\xi}^{\,\on{\acute{e}t}}$.
\end{ex}

It can be somewhat inconvenient to always have to look at geometric points. We can often avoid it as in~\cite{HolmesUniversalJacobian}, by reducing to a case in which the dual graphs already make sense without working \'etale-locally on the base.

\subsection{Quasisplitness, dual graphs at non-geometric points}

\begin{defi}[see~\protect{\cite[Definition~4.1]{HolmesUniversalJacobian}}]
	We say that a nodal curve $X \to S$ is \emph{quasisplit} if the two following conditions are met:
\begin{enumerate}
\item For any point $s\in S$ and any irreducible component $E$ of $X_s$, there is a smooth section $S \to (X/S)^{\mathrm{smooth}}$ intersecting $E$.
\item The singular locus $\sing(X/S)\ra S$ is of the form
	\[
	\coprod\limits_{i\in I}F_i\ra S,
	\]
where the $F_i\to S$ are closed immersions.
\end{enumerate}	
\end{defi}

\begin{ex}
Consider the real conic
\[
X=\operatorname{Proj}(\mathbb R[x,y,z]/(x^2+y^2)).
\]
It is an irreducible nodal curve over $\spec \mathbb R$, but the base change $X_{\mathbb C}$ has two irreducible components: $X$ is not quasisplit over $\spec \mathbb R$.

On the other hand, consider the real projective curve
\[
Y=\operatorname{Proj}(\mathbb R[x,y,z]/(x^3+xy^2+xz^2)).
\]
It has two irreducible components (respectively cut out by $x$ and by $x^2+y^2+z^2$), both geometrically irreducible. The singular locus of $Y/\mathbb R$ consists of two $\C$-rational points, with projective coordinates $(0:i:1)$ and $(0:-i:1)$, at which $Y_{\mathbb C}$ is nodal. Since $\sing(Y/\mathbb R)$ is not a disjoint union of $\mathbb R$-rational points, $Y$ is not quasisplit over $\spec \mathbb R$. However, both $X$ and $Y$ become quasisplit after base change to $\spec\C$.
\end{ex}

\begin{rem}
Our definition of quasisplitness is slightly more restrictive than that of~\cite{HolmesUniversalJacobian}.
\end{rem}

\begin{rem}\label{remark:Zariski_dual_graphs_of_QS_curves}
Let $X/S$ be a quasisplit nodal curve, $s$ a point of $S$ and $\bar s \to S$ a geometric point above $s$. The irreducible components of $X_{\bar s}$ are in canonical bijection with those of $X_s$ by the first condition defining quasisplitness, and the thicknesses of $X$ at $\bar s$ come from principal ideals of the local ring (in the Zariski topology) $\O_{S,s}$ by the second condition. Therefore, we can define without ambiguity the dual graph of $X$ at $s$: its vertices are the irreducible components of $X_s$, its edges are the non-smooth points $x\in X_s$, with endpoints the two components meeting at $x$, and the label of $x$ is the preimage in $\overline{\O_{S,s}}$ of the thickness of some (equivalently, any) point above $x$ in a geometric fibre of $X/S$.

From now on, we will call the label of $x$ defined as above 
the thickness of $X$ at $x$ and talk freely about the dual graphs of quasisplit curves at (not necessarily geometric) field-valued points of $S$. This can clash with Definition \ref{definition:dual_graph} when $x$ is a singular point of a geometric fibre 
of $X/S$. Unless specified otherwise, when there is an ambiguity, we will always privilege Definition \ref{definition:dual_graph}.
\end{rem}

\begin{lem}
Quasisplit curves are stable under arbitrary base change.
\end{lem}

\begin{proof}
The two conditions forming quasisplitness are preserved by base change.
\end{proof}

\begin{lem}\label{lemma:quasisplit_if_base_strictly_local}
Let $S$ be a Noetherian strictly local scheme and $X/S$ a nodal curve. Then $X/S$ is quasisplit.
\end{lem}

\begin{proof}
There is a section through every closed point in the smooth locus of $X/S$, so in particular there is a smooth section through every irreducible component of every fibre. Proposition \ref{structure locale des courbes nodals} implies that the map $\sing(X/S)\ra S$ is a disjoint union of closed immersions.
\end{proof}

\begin{cor}\label{corollary:curves_are_QS_over_some_etale_cover}
Let $S$ be a locally Noetherian scheme and $X/S$ a nodal curve. Then there is an \'etale cover $V \to S$ such that $X_V/V$ is quasisplit.
\end{cor}

\begin{lem}\label{lemma:branches_at_a_singularity_are_Cartier_disjoint}
Let $S$ be a locally Noetherian scheme, $X/S$ a quasisplit nodal curve, $s$ a point of $S$ and $x$ a singular point of $X_s$. Quasisplitness of $X/S$ gives a factorization
\[
x \to F \to \sing(X/S) \to X \to S,
\]
where $F \to S$ is a closed immersion and $F \to \sing(X/S)$ is the connected component containing $x$. Then, there exist an \'etale neighbourhood $(V,y)$ of $x$ in $X$, two effective Cartier divisors $C,D$ on $V$ and an isomorphism $V \times_X F = C \times_V D$ such that $V \times_S F$ is the union of $C$ and $D$.
\end{lem}

\begin{proof}
Let $\bar s$ be a geometric point of $S$ mapping to $s$, and $\bar x=x\times_s \bar s$. By Proposition \ref{structure locale des courbes nodals}, we have an isomorphism $\widehat{\O_{X,\bar x}^{\,\on{\acute{e}t}}}=\widehat{\O_{S,\bar s}^{\,\on{\acute{e}t}}}[[u,v]]/(uv-\Delta)$, where $\Delta$ is a lift to $\O_{S,\bar s}^{\,\on{\acute{e}t}}$ of the thickness of $x$. The base change of $F/S$ to $\spec\O_{S,\bar s}^{\,\on{\acute{e}t}}$ is cut out by $\Delta$, and the zero loci $C_u$ of $u$ and $C_v$ of $v$ are effective Cartier divisors on $\widehat{\O_{X,\bar x}^{\,\on{\acute{e}t}}}$, intersecting in $\widehat{\O_{X,\bar x}^{\,\on{\acute{e}t}}}/(u,v) = F\times_X \spec \widehat{\O_{X,\bar x}^{\,\on{\acute{e}t}}}$. The union of $C_u$ and $C_v$ is $\widehat{\O_{S,\bar s}^{\,\on{\acute{e}t}}}[[u,v]]/(\Delta, uv)$, so the proposition follows by a limit argument.
\end{proof}

\section{Primality, local factoriality and base change}\label{section_intro2}

\subsection{Smooth-factorial schemes}
The main results of this article hold when the base scheme $S$ is quasiexcellent, locally Noetherian and smooth-factorial (\textit{cf.} Definition \ref{definition:smooth_factorial}). In this subsection, we discuss smooth-factoriality and try to give some intuition for it.

\begin{lem}[Popescu's theorem]\label{lemma:Popescu}
Let $R$ be a Noetherian and excellent local ring; then $\widehat{R}$ is a directed colimit of smooth $R$-algebras.
\end{lem}

\begin{proof}
This is a special case of~\cite[\href{https://stacks.math.columbia.edu/tag/07GC}{Tag 07GC}]{stacks-project}.
\end{proof}

\begin{defi}\label{definition:smooth_factorial}
Let $S$ be a scheme. We say that $S$ is \emph{smooth-factorial} (resp.\ \emph{\'etale-factorial}) if any smooth (resp.\ \'etale) $S$-scheme is locally factorial.
\end{defi}

\begin{rem}
Any regular scheme $S$ is smooth-factorial.
\end{rem}

\begin{lem}[\textit{cf.} \protect{\cite[Proposition~1]{Danilov1970On-A-Conjecture}}]\label{lemma:factoriality_descends_under_faithfully_flat_ring_ext}
Let $R\to R'$ be a faithfully flat morphism of Noetherian, integrally closed local rings. Then any ideal $p\subset R$ is principal if $p\otimes_R R'$ is. In particular, if $R'$ is a unique factorization domain, then $R$ is a unique factorization domain.
\end{lem}

\begin{proof}
By faithfully flat descent, $p$ is a finitely generated projective $R$-module of rank $1$, so it is principal.
\end{proof}

\begin{cor}\label{corollary:etale_factorial_equiv_geometrically_factorial}
Let $S$ be a normal and locally Noetherian scheme. Then $S$ is \'etale-factorial if and only if all of its \'etale local rings are unique factorization domains.
\end{cor}

\begin{rem}
In view of Corollary \ref{corollary:etale_factorial_equiv_geometrically_factorial}, \'etale-factoriality is a relatively easy condition to understand and verify, but smooth-factoriality is \textit{a priori} harder to grasp. It seems reasonable to hope that they are equivalent under mild assumptions (\textit{e.g.}\ local Noetherianity). By Lemma \ref{lemma:smooth_factorial_equiv_etale_factorial}, proving this reduces to showing that given a Noetherian, strictly henselian, local unique factorization domain $R$, the strict localizations of $R[X]$ at $\mathfrak m_R$ and at $\mathfrak q:=(\mathfrak m_R,T)$ are unique factorization domains. In~\cite{Danilov1970On-A-Conjecture}, Vladimir Danilov states the related conjecture that $R[[X]]$ must be a unique factorization domain. When $R$ is excellent, the equivalent claims in Lemma \ref{lemma:smooth_factorial_equiv_etale_factorial} imply Danilov's conjecture by Lemma \ref{lemma:Popescu}. Conversely, Danilov's conjecture implies that $(R[X]_{\mathfrak q})^{\mathrm{sh}}=(R[X]_{\mathfrak q})^{\mathrm{h}}$ is a unique factorization domain by Lemma \ref{lemma:factoriality_descends_under_faithfully_flat_ring_ext}.
\end{rem}

\begin{lem}\label{lemma:smooth_factorial_equiv_etale_factorial}
The following three claims are equivalent: 
\begin{enumerate}
\item \label{lem4.7-1} A locally Noetherian scheme $S$ is smooth-factorial if and only if it is \'etale-factorial.
\item \label{lem4.7-2} If $S$ is a locally Noetherian \'etale-factorial scheme, then $\mathbb A^1_S$ is \'etale-factorial. 
\item \label{lem4.7-3} If $R$ is a strictly henselian, Noetherian local ring with $\spec R$ \'etale-factorial, then the strict localizations of $R[X]$ at the prime ideals $\mathfrak m_R$ and at $(\mathfrak m_R,T)$ are unique factorization domains.
\end{enumerate}
\end{lem}

\begin{proof}
We clearly have \eqref{lem4.7-1}$\implies$\eqref{lem4.7-2} and \eqref{lem4.7-2}$\implies$\eqref{lem4.7-3}. Any smooth morphism of schemes $Y \to S$ factors locally as $Y \to \A^n_S \to S$, where $Y \to \A^n_S$ is \'etale, so by induction we have \eqref{lem4.7-2}$\implies$\eqref{lem4.7-1}. For \eqref{lem4.7-3}$\implies$\eqref{lem4.7-2}, suppose that \eqref{lem4.7-3} holds, and consider a locally Noetherian, \'etale-factorial scheme $S$. It suffices to show that the \'etale local ring of $\mathbb A^1_S$ at an arbitrary point $y$ is a unique factorization domain. Denote by $s$ the image of $y$ in $S$. By Lemma~\ref{lemma:factoriality_descends_under_faithfully_flat_ring_ext}, we may assume that $S=\spec R$ is strictly local, with closed point $s$. Translating by a $S$-section of $\mathbb A^1_S$ if necessary, we may assume that the image of $y$ in $\mathbb A^1_s$ is either the generic point or the origin. Therefore, $y$ corresponds to one of the ideals $\mathfrak m_R$ and $(\mathfrak m_R,T)$ of $R[T]$, and we are done.
\end{proof}

\begin{rem}
The claims of Lemma \ref{lemma:smooth_factorial_equiv_etale_factorial} remain equivalent if one removes the Noetherianity assumptions in all three of them.
\end{rem}

\subsection{Permanence of primality under certain morphisms of local rings}

Let $R$ be a strictly henselian local ring such that $\widehat R$ is a unique factorization domain, let $X/\spec R$ be a nodal curve, and let $x$ be a singular closed point of $X$ whose thickness is prime in $\widehat R$. We will see in Lemma~\ref{les anneaux locaux sont UFD} that $X$ is locally factorial at $x$. Since a generic line bundle on a locally factorial scheme always extends to a line bundle, it follows that, in order to construct N\'eron models of Jacobians, we are interested in questions of permanence of primality (of an element of a smooth-factorial ring) under smooth maps, \'etale maps and completions. We discuss these matters in this section.

An element $\Delta$ of an integral local ring $R$ is prime in $R^{\mathrm{sh}}$ when the quotient $R/(\Delta)$ is \emph{geometrically unibranch} (\textit{i.e.}\ its strict henselization is integral, or, equivalently, its normalization is local with purely inseparable residue extension), so we are interested in questions of permanence of geometrically unibranch rings under tensor product. For a more detailed discussion on unibranch rings or counting branches in general, see~\cite[chapitre~IX]{Raynaud} or~\cite[section~23.2]{EGA4.1}. 

In~\cite{Tensor_local}, Moss Eisenberg Sweedler gives a necessary and sufficient condition for the tensor product of two local algebras over a field to be local. We are interested in a sufficient condition for algebras over a strictly local ring. The proof of~\cite{Tensor_local} carries over without much change: we reproduce it here.

\begin{lem}\label{lemme le produit tensoriel reste local}
Let $R$ be a strictly henselian local ring, $R\ra A$ an integral morphism of local rings with purely inseparable residue extension and $R\ra B$ any morphism of local rings. Then $A\otimes_R B$ is local, and its residue field is purely inseparable over that of $B$.
\end{lem}

\begin{proof}
Let $\m$ be a maximal ideal of $A\otimes_R B$. The map $B\ra~A\otimes_R B$ is integral, so it has the going-up property (\textit{cf.} \cite[\href{https://stacks.math.columbia.edu/tag/00GU}{Tag 00GU}]{stacks-project}); therefore, the inverse image of $\m$ in $B$ is a maximal ideal: it must be $\m_B$. Thus $\m$ contains $A\otimes_R\m_B$.

In particular, $\m$ also contains the image of $\m_R$ in $A\otimes_R B$: it corresponds to a maximal ideal of \mbox{$(A\otimes_R B)/(\m_R A\otimes_R B)$,} which we will still call $\m$. We have a commutative diagram
\[
\xymatrix{
k_R\ar[r]\ar[d]&B/\m_R B\ar[d]\\
A/\m_R A\ar[r]&(A\otimes_R B)/(\m_R A\otimes_R B)\rlap{.}
}
\]
Since $A/\m_R A$ is local and integral over the field $k_R$, its maximal ideal $m_A$ is nilpotent and is its only prime ideal. The inverse image of $\m$ in $A/\m_R A$ is a prime ideal, so it can only be $\m_A$. This shows that, as an ideal of $A\otimes_R B$, $\m$ also contains $\m_A\otimes_R B$.

Every maximal ideal of $A\otimes_R B$ contains both $\m_A\otimes_R B$ and $A\otimes_R\m_B$, so the maximal ideals of $A\otimes_R B$ are in bijective correspondence with those of $k_A\otimes_{k_R} k_B=A\otimes_R B/(\m_A\otimes_R B+A\otimes_R\m_B)$. We will now show that the latter is local, with purely inseparable residue extension over $k_B$.

By hypothesis, the extension $k_A/k_R$ is purely inseparable. If $k_R$ has characteristic $0$, then $k_A=k_R$, and we are done. Suppose $k_R$ has characteristic $p>0$. For any $x\in k_A\otimes_{k_R} k_B$, we can write $x$ as a finite sum $\sum_{i\in I}\lambda_i\otimes\mu_i$ with the $\lambda_i,\mu_i$ in $k_A,k_B$, respectively. There is an integer $N>0$ such that for all $i$, $\lambda_i^{p^N}$ is in $k_R$. Therefore, $x^{p^N}=\sum_{i\in I}\lambda_i^{p^N}\mu_i^{p^N}$ is in $k_B$, so $x$ is either nilpotent or invertible. It follows that $k_A\otimes_{k_R} k_B$ is local, with maximal ideal its nilradical, and that its residue field is purely inseparable over $k_B$, as claimed.
\end{proof}

\begin{lem}\label{lemme unibranche implique unibranche apres morphisme lisse local}
Let $(R,\m)$ be an integral and strictly local Noetherian ring. Let $R\ra R'$ be a smooth ring map, let $\p$ be a prime ideal of $R'$ containing $\m R'$, and let $(R'_p)^{\mathrm{sh}}$ be a strict henselization of $R'_\p$. Then $R'_\p$ is geometrically unibranch; \textit{i.e.}\ $(R'_p)^{\mathrm{sh}}$ is an integral domain.
\end{lem}

\begin{proof}
We know $(R'_p)^{\mathrm{sh}}$ is reduced since it is a filtered colimit of smooth $R$-algebras. Let $B,B'$ be the integral closures of $R,R'_\p$ in their respective fraction fields. The ring $R'_\p$ is integral, so by~\cite[corollaire~IX.1]{Raynaud}, $(R'_p)^{\mathrm{sh}}$ is an integral domain if and only if $B'$ is local and the extension of residue fields of $R'_\p\ra B'$ is purely inseparable. But any smooth base change of $B/R$ remains normal (see~\cite[Corollary 8.2.25]{Liu}), so $B\otimes_R R'_\p$ is normal as a filtered colimit of normal $B$-algebras. Moreover, any normal algebra over $R'_\p$ must factor through $B\otimes_R R'_\p$, so we have $B'=B\otimes_R R'_\p$. Applying Lemma \ref{lemme le produit tensoriel reste local}, we find that $B'$ is local and the extension of residue fields of $R'_\p\ra B'$ is purely inseparable, which concludes the proof.
\end{proof}

\begin{cor}\label{corollary:irred_of_etale_loc_ring_stable_under_smooth_maps}
Let $S$ be a smooth-factorial scheme and $Y\to S$ a smooth morphism, and consider a commutative square
\[
\xymatrix{
y \ar[r]\ar[d] 	& s \ar[d] \\
Y \ar[r]		& S\rlap{,}
}
\]
where $s \to S$ and $y \to Y$ are geometric points. Then for any prime element $\Delta$ of $\O_{S,s}^{\,\on{\acute{e}t}}$, the image of $\Delta$ in $\O_{Y,y}^{\,\on{\acute{e}t}}$ is prime.
\end{cor}

\begin{proof}
Base change to $\spec\O_{S,s}^{\,\on{\acute{e}t}}/(\Delta)$, replace $Y$ by an affine neighbourhood of $y$ in $Y$, and apply Lemma~\ref{lemme unibranche implique unibranche apres morphisme lisse local}.
\end{proof}

\begin{lem}\label{lemme les irreductibles au complete sont les irreductibles}
	Let $R$ be a strictly henselian and excellent local ring. Then an element $\Delta$ of $R$ is prime in $R$ if and only if it is prime in $\widehat{R}$.
\end{lem}

\begin{proof}
The nontrivial implication is the "only if" part. Suppose $\Delta$ is prime in $R$. By Lemma \ref{lemma:Popescu}, $\widehat{R}$ is a directed colimit of smooth $R$-algebras. Therefore, $R/(\Delta) \to \widehat{R}/(\Delta)$ is a colimit of smooth $R/(\Delta)$-algebras, and we conclude by Lemma~\ref{lemme unibranche implique unibranche apres morphisme lisse local}.
\end{proof}

\begin{lem}\label{lemme existence des sections avant completion pour raffinements asymetriques}
Let $S$ be an excellent and smooth-factorial scheme, $X/S$ a $S$-scheme of finite presentation, $\bar s$ a geometric point of $S$ and $x$ a closed point of $X_{\bar s}$ with an isomorphism $\widehat{\O_{X,x}^{\,\on{\acute{e}t}}}=\widehat{\O_{S,\bar s}^{\,\on{\acute{e}t}}}[[u,v]]/(uv-\Delta)$ for some $\Delta\in\m_{\bar s}\subset \O_{S,{\bar s}}^{\,\on{\acute{e}t}}$. For every $t_1,t_2\in \m_{\bar s}$ such that $t_1t_2=\Delta$, there exist an \'etale neighbourhood $S' \to S$ of ${\bar s}$ and a section $S' \to X$ through $x$ such that the induced map $\widehat{\O_{X,x}^{\,\on{\acute{e}t}}}\ra\widehat{\O_{S',{\bar s}}^{\,\on{\acute{e}t}}}$ sends $u,v$ to generators of $(t_1)$ and $(t_2)$, respectively.
\end{lem}

\begin{proof}
Put $R=\O_{S,{\bar s}}^{\,\on{\acute{e}t}}$. Then $\widehat{R}$ is a unique factorization domain by Lemma~\ref{lemma:Popescu}. Consider the map $\widehat{\O_{X,x}^{\,\on{\acute{e}t}}}\ra\widehat{R}$ that sends $u,v$ to $t_1,t_2$, respectively. Compose it with $\O_{X,x}^{\,\on{\acute{e}t}}\ra\widehat{\O_{X,x}^{\,\on{\acute{e}t}}}$ to get a map $f_0:\O_{X,x}^{\,\on{\acute{e}t}}\ra\widehat{R}$.

For Noetherian local rings, quotients commute with completion with respect to the maximal ideal, so two distinct ideals are already distinct modulo some power of the maximal ideal. Let $\prod_{i=1}^n\Delta_i^{\nu_i}$ be the prime factor decomposition of $\Delta$ in $R$. Principal ideals of $R$ of the form $(\Delta_i^{\mu_i})$ with $0\leq\mu_i\leq\nu_i$ are pairwise distinct and in finite number, so there exists some $N\in\N$ such that their images in $R/\m_R^N$ are pairwise distinct. Since $R$ is henselian and excellent, it has the Artin approximation property (\textit{cf.} \cite[\href{https://stacks.math.columbia.edu/tag/07QY}{Tag 07QY}]{stacks-project}), so there exists a map $f\colon\O_{X,x}^{\,\on{\acute{e}t}}\ra R$ which coincides with $f_0$ modulo $m_R^N$. This $f$ induces a map $\widehat{f}:\widehat{\O_{X,x}^{\,\on{\acute{e}t}}}\ra\widehat{R}$. Denote by $a,b$ the respective images of $u,v$ by $\widehat{f}$; we have $a=t_1$ and $b=t_2$ in $R/\m_R^N$. But $ab=\Delta$ in $\widehat{R}$ and, by Lemma~\ref{lemme les irreductibles au complete sont les irreductibles}, $\Delta$ has the same prime factor decomposition in $R$ and $\widehat{R}$, so the only principal ideals of $\widehat{R}$ containing $\Delta$ are of the form $(\Delta_i^{\mu_i})$ with $0\leq\mu_i\leq\nu_i$. By the definition of $N$, we get $a\widehat{R}=t_1\widehat{R}$ and $b\widehat{R}=t_2\widehat{R}$. Since $X/S$ is finitely presented, $f$ comes from an $S$-morphism $S' \to X$, where $S'$ is an \'etale neighbourhood of ${\bar s}$ in~$S$.
\end{proof}

\section{Sections of nodal curves}\label{section_intro3}

We present some technicalities regarding sections of nodal curves, with a view towards studying those morphisms that are locally the blow-up in the ideal sheaf of a section.
The basis for this formalism was thought of together with Giulio Orecchia.

\subsection{Type of a section}

We will define a combinatorial invariant, the \emph{type} of a section, summarizing information about the behaviour of the said section around the singular locus of a nodal curve $X/S$. Later on, we will show that sections of all types exist \'etale-locally on the base (\textit{cf.} Proposition~\ref{proposition:admissibles_are_a_basis}) and that the type of a section locally characterizes the blow-up of $X$ in the ideal sheaf of that section (Corollary~\ref{corollary:unicite_zariski_locale_des_T-refinements}).

\begin{defi}\label{definition:orientation}
Let $S$ be a locally Noetherian scheme, $X/S$ a quasisplit nodal curve, $s$ a point of $S$ and $x$ a singular point of $X_s$. Let $F$ be the connected component of $\sing(X/S)$ containing $x$. Then the set of connected components of $(X\backslash F)\times_X \spec \O_{X,x}^{\,\on{\acute{e}t}}\times_S F$ is a pair $\{C,D\}$ (see Proposition \ref{structure locale des courbes nodals} and Lemma \ref{lemma:branches_at_a_singularity_are_Cartier_disjoint}), on which the Galois group $\aut_{O_S}(\O_{S,s}^{\,\on{\acute{e}t}})=\gal(k(s)^{\mathrm{sep}}/k(s))$ acts. If this action is trivial, we say $X/S$ is \emph{orientable at $x$}, and we call \emph{orientations of $X/S$ at $x$} the ordered pairs $(C,D)$ and $(D,C)$. The scheme-theoretic closures of $C$ and $D$ in $\spec \O_{X,x}^{\,\on{\acute{e}t}}$ are effective Cartier divisors, and we will often also call them $C$ and $D$.

If $X/S$ is orientable at a singular point $x$, and if $x'$ is a singular point specializing to $x$, then $X/S$ is orientable at $x'$ and there is a canonical bijection between orientations at $x$ and at $x'$. Given an orientation $(C_1,C_2)$ at~$x$, we will often still write $(C_1,C_2)$ for the induced orientation at $x'$.

We say $X/S$ is \emph{orientable} if it admits orientations at all points, compatible with the generization isomorphisms between orientations. In that case, we call \emph{global orientation} (or just \emph{orientation}) of $X/S$ a compatible family $(C_{1,x},C_{2,x})_{x\in\sing(X/S)}$, where $(C_{1,x},C_{2,x})$ is an orientation at $x$. We will often abusively write global orientations as ordered pairs $(C_1,C_2)$ and confuse them with the induced orientation at any given point $x\in\sing(X/S)$.
\end{defi}

\begin{rem}
If $X/S$ is orientable, then it is orientable at every point, but the converse is not true in general.
\end{rem}

\begin{rem}\label{remark:orientations_are_choices_of_preimage_in_normalization}
The curve $X/S$ is orientable at $x$ if and only if the preimage of $x$ in the normalization of $X_s$ consists of two $k(s)$-rational points, in which case an orientation is the choice of one of these points. Roughly speaking, this also corresponds to picking an orientation of the edge corresponding to $x$ in the dual graph of $X$ at $s$. The "roughly speaking" is due to the case of loops: there is an ambiguity on how to orient them. We could get rid of this ambiguity by using a heavier notion of dual graphs (such as the tropical curves often used in log geometry), but this work does not require it.
\end{rem}

\begin{lem}\label{lemma:orientations_exist_locally}
Let $S$ be a locally Noetherian scheme and $X/S$ a quasisplit nodal curve. Then, there exists an \'etale cover $V \to S$ such that $X_V/V$ is orientable.
\end{lem}

\begin{proof}
It suffices to show that any $s \in S$ has an \'etale neighbourhood over which $X$ is orientable, which follows from observing that $X/S$ is finitely presented and that a nodal curve over a strictly local scheme is orientable.
\end{proof}

\begin{lem}\label{lemma:orientable_curves_stable_under_basechange}
Let $S' \to S$ be a morphism between locally Noetherian schemes. Let $X/S$ be a quasisplit nodal curve. If $X/S$ is orientable at a point $x\in X$, then $X':=X\times_S S'$ is orientable at any singular point $x'$ above $x$, and orientations of $X$ at $x$ naturally pull back to orientations of $X'$ at $x'$. In particular, if $X/S$ is orientable, then $X'/S'$ is orientable.
\end{lem}

\begin{proof}
This follows from Remark \ref{remark:orientations_are_choices_of_preimage_in_normalization}.
\end{proof}

\begin{defi}[Type of a section]\label{definition:type_of_section}
Let $X/S$ be a quasisplit nodal curve with $S$ smooth-factorial. Suppose $X$ is smooth over a dense open subscheme $U$ of $S$. Let $s$ be a point of $S$ and $x$ a singular point of $X_s$. We call \emph{type at $x$} any element of the monoid $\overline{\O_{S,s}^{\,\on{\acute{e}t}}}$ strictly comprised between $1$ and the thickness of $x$ (for the order induced by divisibility). There are only finitely many types at $x$.

Suppose that $X/S$ admits a global orientation $(C_1,C_2)$. Pick an isomorphism
\[
\widehat{\O_{X,x}^{\,\on{\acute{e}t}}}=\widehat{\O_{S,s}^{\,\on{\acute{e}t}}}[[u,v]]/(uv-\Delta_x),
\]
where $C_1$ corresponds to $u=0$ and $\Delta_x\in\O_{S,s}^{\,\on{\acute{e}t}}$ maps to the thickness of $x$ in $\overline \O_{S,s}^{\,\on{\acute{e}t}}$. Let $\sigma$ be a section of $X/S$ through $x$. It induces a morphism
\[
\widehat{\sigma^\#} \colon \widehat{\O_{S,s}^{\,\on{\acute{e}t}}}[[u,v]]/(uv-\Delta_x) \to \widehat{\O_{S,s}^{\,\on{\acute{e}t}}}. 
\]
By Lemma \ref{lemme les irreductibles au complete sont les irreductibles}, $\Delta_x$ has the same prime factor decomposition in $\O_{S,s}^{\,\on{\acute{e}t}}$ and in $\widehat{\O_{S,s}^{\,\on{\acute{e}t}}}$, so there is a canonical embedding of the submonoid of $\widehat{\O_{S,s}^{\,\on{\acute{e}t}}}/(\widehat{\O_{S,s}^{\,\on{\acute{e}t}}})^*$ generated by the factors of $\Delta_x$ into $\overline{\O_{S,s}^{\,\on{\acute{e}t}}}$. We call \emph{type of $\sigma$ at $x$ relative to $(C_1,C_2)$} the image of $u$ in $\overline{\O_{S,s}^{\,\on{\acute{e}t}}}$. It is a type at $x$ and does not depend on our choice of isomorphism $\widehat{\O_{X,x}^{\,\on{\acute{e}t}}}=\widehat{\O_{S,s}^{\,\on{\acute{e}t}}}[[u,v]]/(uv-\Delta_x)$ as long as $C_1$ is given by $u=0$. When they are clear from context, we will omit $x$ and $(C_1,C_2)$ from the notation and just call it the \emph{type of $\sigma$}. In general, given a type~$T$ at $x$, there need not exist a section of type $T$.
\end{defi}

\begin{lem}\label{lemma:type_passes_to_generization}
Let $X/S$, $s$, $x$, $(C_1,C_2)$ and $U$ be as in Definition~\ref{definition:type_of_section} and $\sigma$ be a section $S\to X$ of type $T$ at $x$. Let $s'$ be a generization of $s$. Then there is a singular point of $X_{s'}$ specializing to $x$ if and only if the thickness of $x$ does not map to $1$ in $\overline{\O_{S,s'}^{\,\on{\acute{e}t}}}$. Suppose it is the case, and write $x'$ for this singular point. Then:
\begin{itemize}
\item If the image of $T$ in $\overline{\O_{S,s'}^{\,\on{\acute{e}t}}}$ is either $1$ or the thickness of $x'$, then $\sigma(s')$ is a smooth point of $X_{s'}$.
\item Otherwise, the image of $T$ is a type at $s'$, which we still denote by $T$, and $\sigma$ is of type $T$ at $x'$ relative to $(C_1,C_2)$.
\end{itemize}
\end{lem}

\begin{proof}
By Proposition \ref{structure locale des courbes nodals}, the set of non-smooth points of $X_{s'}/s'$ specializing to $x$ is empty if the thickness of $x$ maps to $1$ in $\overline{\O_{S,s'}^{\,\on{\acute{e}t}}}$, and it is a singleton $\{x'\}$ otherwise. Suppose the latter holds; then we conclude using Proposition \ref{structure locale des courbes nodals} and the definition of the type of a section.
\end{proof}

\begin{rem}
One can think of the thickness of $x$ as the relative version of a length, and of the type of a section $\sigma$ relative to an orientation $(C_1,C_2)$ as a measure of the intersection of $\sigma$ with $C_1$, seen as an effective Cartier divisor locally around $x$ as in Lemma \ref{lemma:branches_at_a_singularity_are_Cartier_disjoint}. In other words, the type is a measure of "how close to $C_1$" the section is.
\end{rem}

\begin{prop}\label{proposition:locus_of_same_type_is_open}
Let $S$ be a smooth-factorial and excellent scheme, and let $X/S$ be a quasisplit nodal curve, smooth over some dense open subscheme $U$ of $S$. Let $\sigma$ and $\sigma'$ be two $S$-sections of $X$. Then the union of $(X/S)^{\mathrm{smooth}}$ with the set of singular points $x$ of $X/S$ at which $\sigma$ and $\sigma'$ have the same type is an open subscheme of $X$, which we call the \emph{same type locus} of $\sigma$ and $\sigma'$.
\end{prop}

\begin{proof}
Since the smooth locus of $X/S$ is open in $X$, the proposition reduces to showing that if $\sigma$ and $\sigma'$ have the same type at a singular point $x$ of $X$, then they have the same type at every singular point in an open neighbourhood of $x$. Let $s$ be the image of $x$ in $S$. Pick an isomorphism
\[
\widehat{\O_{X,x}^{\,\on{\acute{e}t}}}=\widehat{\O_{S,s}^{\,\on{\acute{e}t}}}[[u,v]]/(uv-\Delta_x),
\]
where $\Delta_x\in \O_{S,s}$ is a lift of the thickness of $x$. By hypothesis, the images $\Delta,\Delta'$ of $u$ under the two morphisms $\widehat{\O_{X,x}^{\,\on{\acute{e}t}}} \to \widehat{\O_{S,s}^{\,\on{\acute{e}t}}}$ given by $\sigma$ and $\sigma'$ are the same up to a unit $\lambda\in \widehat{\O_{S,s}^{\,\on{\acute{e}t}}}^\times$. By Lemma \ref{lemma:Popescu}, base changing to a smooth neighbourhood of $s$ in $S$, we may assume that $\Delta,\Delta',\Delta_x$ come from global sections of $\O_S$ and $\lambda$ from a global section of $\O_S^\times$. Pick a smooth neighbourhood $W$ of $x$ in $X$ such that $u,v$ come from global sections of $W$. Shrinking $S$, we may assume that $\sigma,\sigma'$ factor through $W$ and that their comorphisms map $u$ to $\Delta$ and $\Delta'$, respectively. Shrinking further, we may assume that the non-smooth locus of $W/S$ is cut out by $(u,v,\Delta_x)$. Then, the image of $W$ in $X$ is a Zariski open neighbourhood of $x$ contained in the same type locus of $\sigma$ and~$\sigma'$.
\end{proof}


\subsection{Admissible neighbourhoods}

Here we show that when one works \'etale-locally on the base (in a sense that we will make precise), one can always assume that sections of all types exist.

\begin{defi}\label{definition:admissible}
Let $S$ be a smooth-factorial scheme and $X/S$ a nodal curve, smooth over a dense open $U$ of $S$. Let $s$ be a point of $S$ and $(V,v)$ an \'etale neighbourhood of $s$ in $S$. We say that $(V,v)$ is an \emph{admissible neighbourhood of $s$} (relative to $X/S$) when the following conditions are met:
\begin{enumerate}
\item \label{def5.10-1} The curve $X_V/V$ is quasisplit and orientable.
\item \label{def5.10-2} For any singular point $x$ of $X_v$, every prime factor in $\overline{\O_{S,s}^{\,\on{\acute{e}t}}}$ of the thickness of $x$ lifts to a global section of $\O_V$.
\item \label{def5.10-3} For every singular point $x$ of $X_v$ (however oriented), there are sections $V\to X_V$ of all types at $x$.
\end{enumerate}
When $\bar s \to S$ is a geometric point with image $s$ and $(V,v)$ an admissible neighbourhood of $s$ with a factorization $\bar s \to v$, we will also sometimes call $V$ an \emph{admissible neighbourhood of $\bar s$}.
\end{defi}

\begin{rem}\label{remark:strictly_local_schemes_are_admissible}
In the situation of Definition \ref{definition:admissible}, if $S$ is strictly local, then it is an admissible neighbourhood of its closed point.
\end{rem}

\begin{prop}\label{proposition:admissibles_are_a_basis}
Let $X/S$ be a nodal curve, where $S$ is a smooth-factorial and excellent scheme. Then any point $s\in S$ has an admissible neighbourhood.
\end{prop}

\begin{proof}
All three conditions in the definition of admissibility are stable under base change to an \'etale neighbourhood of $s$. We can assume that $X/S$ is quasisplit by Corollary \ref{corollary:curves_are_QS_over_some_etale_cover} and that it is orientable by Lemma~\ref{lemma:orientations_exist_locally}. Let $x_1,\ldots,x_n$ be the singular points of $X_s$. The thickness of each $x_i$ has only finitely many prime factors in $\overline{\O_{S,s}^{\,\on{\acute{e}t}}}$, so we can shrink $S$ again into a neighbourhood satisfying condition~\eqref{def5.10-2} of the definition of admissibility. The fact that this $V$ can be shrunk again until it meets all three conditions follows from Lemma \ref{lemme existence des sections avant completion pour raffinements asymetriques}.
\end{proof}

\begin{rem}
If $(V,v)$ is an admissible neighbourhood of a point $s$ of $S$, then $V$ is not necessarily admissible even at generizations of $v$ (see Example \ref{example:irred_not_etale_irred}). Thus, it is not easy \textit{a priori} to find a good global notion of admissible cover.
\end{rem}

\begin{ex}\label{example:irred_not_etale_irred}
Let $R=\spec\C[[u,v,w]]$. Then $R$ is regular (hence smooth-factorial), local and strictly henselian. The element $\Delta:=u^2(v-w)-v^2(u+w)$ is prime in $R$. Let $X_K$ be the elliptic curve over $K:=\Frac R$ cut out in $\P^2_K$ by the equation $y^2=(x-1)(x^2-\Delta)$ (in affine coordinates $x,y$). The minimal Weierstrass model of $X_K$ is a nodal curve $X$ over $\spec R$, whose closed fibre has exactly one singular point with label $\Delta$. Let $t$ be the point of $S:=\spec R$ corresponding to the prime ideal $(u,v)$ of $R$. Then $\Delta$ has two prime factors in the \'etale local ring of $S$ at $t$ (because the invertible elements $v-w$ and $u+w$ of this \'etale local ring admit square roots). Hence, $S$ is an admissible neighbourhood of its closed point but not of $t$ (relative to $X/S$).
\end{ex}

The next proposition states that admissible neighbourhoods behave well with respect to the \emph{smooth} topology.

\begin{prop}\label{proposition:admissibles_compatible_with_smooth_basechange}
Let $S$ be a smooth-factorial scheme and $X/S$ a quasisplit nodal curve, smooth over some dense open subscheme $U$ of $S$. Let $\pi\colon Y\to S$ be a smooth morphism, $y$ a point of $Y$ and $s=\pi(y)$. Let $V$ be an admissible neighbourhood of $s$ in $S$; then $V\times_S Y$ is an admissible neighbourhood of $y$ in $Y$.
\end{prop}

\begin{proof}
This follows from Corollary \ref{corollary:irred_of_etale_loc_ring_stable_under_smooth_maps} and the definition of admissible neighbourhoods.
\end{proof}

\section{Relating different nodal models}\label{section1}
This section is dedicated to constructing inductively nodal models of a smooth curve with prime thicknesses, starting from any nodal model.

\subsection{Arithmetic complexity and motivation for refinements}

\begin{defi}\label{complexite arithmétique}
	Let $M$ be the free commutative monoid over a set of generators $G$. We call \emph{arithmetic complexity} of $m\in M\backslash \{0\}$ the integer $n-1$, where $n$ is the (unique) $n\in\N^*$ such that we can write $m=\prod_{i=1}^n g_i$ with all the $g_i$ in $G$.	Given a graph $\Gamma$ labelled by $M$, we define the arithmetic complexity of an edge to be that of its label, and the arithmetic complexity of $\Gamma$ to be the sum of the arithmetic complexities of its edges. Given a nodal curve $X/S$, where $S=\spec R$ is a local unique factorization domain, the monoid $\overline R$ of principal ideals of $R$ is freely generated by the prime principal ideals. From now on, when we talk about arithmetic complexities of edges of dual graphs, we will always be referring to this set of generators. We define the arithmetic complexity of $X$ at a geometric point $s \to S$ as that of its dual graph at $s$ and the arithmetic complexity of a singular point $x \in X_s$ as that of the corresponding edge. We give similar definitions when $X/S$ is quasisplit and $s\in S$ is a point.
	
	Note that $X$ has arithmetic complexity $0$ if and only if every point has prime thickness: it is an integer measuring "how far away from being prime" the thicknesses are.
\end{defi}

The following lemma essentially shows that nodal curves are locally factorial around their singular points that have prime thicknesses. In particular, any generic line bundle extends locally around such a point.

\begin{lem}\label{les anneaux locaux sont UFD}
	Let $R$ be a complete and local unique factorization domain and $\Delta$ be an element of $\m_R$. Then $\widehat{A}:=R[[u,v]]/\left(uv-\Delta\right)$ is a unique factorization domain if and only if $\Delta$ is prime in $R$.
\end{lem}

\begin{proof}
  Suppose that $\widehat{A}$ is a unique factorization domain, and let $d$ be a prime factor of $\Delta$ in $R$. Denote by $S$ the complement of the prime ideal $(u,d)$ in $\widehat{A}$. Let $\mathfrak p$ be a nonzero prime ideal of $S^{-1}\widehat A$. Then, $p$ contains a nonzero element $x=ux_u+x_v$, with $x_u$ and $x_v$ in $R[[u]]$ and $R[[v]]$, respectively. Since $\mathfrak p \neq S^{-1}\widehat A$, we have $d|x_v$. Let $n$ and $m$ be the maximal elements of $\N^*\cup\{+\infty\}$ such that $u^n|ux_u$ and $d^m|x_v$. Since $x$ is nonzero, we know that either $n$ or $m$ is finite. If $n\leq m$, then $v^nx=\Delta^n\frac{x_u}{u^{n-1}}+d^n\frac{v^nx_v}{d^n}$  is in $\mathfrak p$ and is associated to 
  $d^n$ in $S^{-1}\widehat{A}$, so we obtain $d\in\mathfrak p$, from which it follows that $\mathfrak p=(u,d)$. If $m<n$, a similar argument shows that $\mathfrak p$ contains $u^m$ and thus equals $(u,d)$. Therefore, $S^{-1}\widehat{A}$ has Krull dimension $1$, \textit{i.e.}\ $(u,d)$ has height $1$ in $\widehat{A}$. Since $\widehat{A}$ is a unique factorization domain, it follows that $(u,d)$ is principal in it, from which we deduce that~$\Delta$ and $d$ are associated in $\widehat{A}$. In particular, $\Delta$ is prime in $R$.
	
	The interesting part is the converse: let us assume that $\Delta$ is prime in $R$. We want to show that $\widehat{A}$ is a unique factorization domain. We first prove that $A:=R[u,v]/(uv-\Delta)$ is a unique factorization domain: let $p$ be a prime ideal of $A$ of height $1$; we have to show that $p$ is principal in $A$. We observe that $u$ is a prime element of $A$ since the quotient $A/(u)=R/(\Delta)[v]$ is an integral domain. Therefore, if $p$ contains $u$, then $p=(u)$ is principal. Otherwise, $p$ gives rise to a prime ideal of height $1$ in $A_u:=A[u^{-1}]$, which is principal since $A[u^{-1}]=R[u,u^{-1}]$ is a unique factorization domain. In that case, write $pA_u=fA_u$ for some $f\in A_u$. Multiplying by a power of the invertible element $u$ of $A_u$, we can choose the generator $f$ to be in $A\backslash uA$. Since $p$ is a prime ideal of $A$ not containing $u$, we know $p$ contains $f$ and thus $fA$. We will now prove the reverse inclusion. Let $x$ be an element of $p$. The localization $pA_u=fA_u$ contains $x$, so $x$ satisfies a relation of the form $u^nx=fy$ for some $n\in\N$ and some $y\in A$. But since $u$ is prime in $A$, we know that $u^n$ divides $y$ and $x$ is in $fA$.
	
	Now, we will deduce the factoriality of $\widehat{A}$ from that of $A$. The author would like to thank Ofer Gabber for providing the following proof. Let $q$ be a prime ideal of $\widehat{A}$ of height $1$; we will show that $q$ is principal. We put $S=\spec R$, $X=\spec A$, $\widehat{X}=\spec \widehat{A}$ and $Z=\spec (\widehat{A}/q)$, so that $Z$ is a prime Weil divisor on $\widehat{X}$. Let $\eta,\eta'$ be the generic points of the respective zero loci of $u,v$ in the closed fibre $\spec k_R[[u,v]]/(uv)$ of $\widehat{X}\to S$. Since $u$ and $v$ are prime elements of $\widehat{A}$, we can once again assume that $Z$ contains neither $\eta$ nor~$\eta'$. It follows that the closed fibre of $Z \to S$ is of dimension $0$: the morphism $Z\to S$ is quasifinite, hence finite by~\cite[section~0.7.4]{EGA1}. A fortiori, $\widehat{A}/q$ is finite over $A$, so by Nakayama's lemma, the morphism $A \to \widehat{A}/q$ is surjective. Denote by $p$ its kernel. Then $A/p=\widehat{A}/q$ is $\mathfrak{m}_A$-adically complete and separated, so it maps isomorphically to its completion $\widehat{A}/p\widehat{A}$. The prime ideal $p$ of $A$ is of height $1$ since $\widehat{X}\to X$ is a flat map of normal Noetherian schemes. Therefore, $p$ is principal in $A$, and $q=p\widehat{A}$ is principal in $\widehat{A}$.
\end{proof}

\subsection{Refinements of graphs}
In order to reap the benefits from the properties of nodal models with prime labels, all we need is an algorithm that takes an arbitrary nodal model as an input and returns another one with strictly lower arithmetic complexity.

\begin{defi}\label{definition refinements of graphs}
	As in~\cite[Definition 3.2]{Holmes}, for a graph $\Gamma=(V,E,l)$ with edges labelled by elements of a monoid $M$, we call \emph{refinement} of $\Gamma$ the data of another labelled graph $\Gamma'=(V',E',l')$ labelled by $M$ and two maps
\begin{align*}
E' & \to E, \\
V' & \to E \coprod V
\end{align*}	
	such that:
	\begin{itemize}
	\item Every vertex $v$ in $V$ has a unique preimage $v'$ in $V'$; 
	\item For every edge $e\in E$ with endpoints $v_1,v_2\in V$, there is a chain $C(e)$ from $v_1',v_2'$ in $\Gamma'$ such that the preimage of $\{e\}$ in $V' \coprod E'$ consists of all edges and intermediate vertices of $C(e)$; 
	\item For all $e\in E$, the length of $e$ is the sum of the lengths of all edges of $C(e)$.
	\end{itemize}
	We will often keep the maps implicit in the notation, in which case we call $\Gamma'$ a \emph{refinement} of $\Gamma$ and write $\Gamma'\preceq\Gamma$. We say $\Gamma'$ is a \emph{strict refinement} of $\Gamma$ and write $\Gamma'\prec\Gamma$ if, in addition, the map $E' \to E$ is not bijective.
\end{defi}

\begin{rem}
Informally, a refinement of a graph is obtained by "replacing every edge by a chain of edges of the same total length". Suppose $\Gamma'\preceq\Gamma$; then $\Gamma'\prec\Gamma$ if and only if at least one of the chains $C(e)$ is of length at least $2$, \textit{i.e.}\ if and only if $\Gamma'$ has strictly more edges than $\Gamma$.
\end{rem}

Now we want to blow up $X$ in a way that does not affect $X_U$ but refines the dual graph. We will define \emph{refinements} of curves (\textit{cf.} Definition \ref{definition refinement}). We will show that, \'etale-locally on the base, any refinement of a dual graph of $X$ comes from a refinement of curves.

\subsection{Refinements of curves}

\begin{lem}\label{structure des (C,delta)-blowups}
Let $f\colon X\ra S$ be a quasisplit nodal curve with $S$ smooth-factorial and excellent. Suppose $X$ is smooth over some dense open $U\subset S$. Let $\sigma\colon S\to X$ be a section and $\phi\colon X'\to X$ be the blow-up in the ideal sheaf of $\sigma$. 

Then $\phi$ is an isomorphism above the complement in $X$ of $\,\sing(X/S)\cap\sigma(S)$. In particular, it is an isomorphism above the smooth locus of $X/S$, which contains $X_U$, so $X'$ is a model of $X_U$. 

Moreover, $X'/S$ is a nodal curve, and its dual graphs are refinements of those of $X$. More precisely, let $s$ be a point of $S$, and suppose $\sigma(s)$ is a singular point $x$ of $X_s$. Choose an orientation $(C,D)$ of $X_{\O_{S,s}^{\,\on{\acute{e}t}}}$ at $x$. Let $T_x$ be the thickness of $x$. In $\overline{\O_{S,s}^{\,\on{\acute{e}t}}}$, write
\[
T_x=TT',
\]
where $T$ is the type of $\sigma$ at $x$ relative to $(C,D)$. Let $\Gamma,\Gamma'$ be the respective dual graphs of $X$ and $X'$ at $s$, and let~$e$ be the edge of $\,\Gamma$ corresponding to $x$. Then $e$ has label $T_x$, and one obtains $\Gamma'$ from $\Gamma$ as follows:

\begin{itemize}
\item If $e$ is not a loop, then $C$ and $D$ come from two distinct irreducible components of $X_s$ $($that we still call $C$ and $D)$. In that case, $\Gamma'$ is obtained from $\Gamma$ by replacing $e$ by a chain

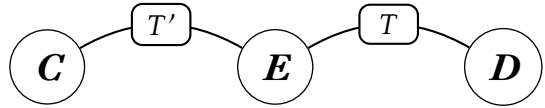
\begin{center}
\begin{tikzpicture}
  \SetGraphUnit{3}
  \Vertex{E}
  \WE(E){C}
  \EA(E){D}
  \Edge[label = $T'$](C)(E)
  \Edge[label = $T$](E)(D)
\end{tikzpicture}
\end{center}

where the strict transforms of $C$ and $D$ are still called $C$ and $D$, and where $E$ is the inverse image of $x$.

\item If $e$ is a loop, \textit{i.e.}\ $x$ belongs to only one irreducible component $L$ of $X_s$, then $\Gamma'$ is obtained from $\Gamma$ by replacing $e$ by a cycle

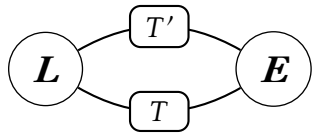
\begin{center}
\begin{tikzpicture}
  \SetGraphUnit{3}
  \Vertex{L}
  \EA(L){E}
  \Edge[label = $T'$](L)(E)
  \Edge[label = $T$](E)(L)
\end{tikzpicture}
\end{center}

where the strict transform of $L$ is still called $L$ and $E$ is the inverse image of $x$.

\end{itemize}

\end{lem}

\begin{proof}
The ideal sheaf of $\sigma$ is already Cartier above the smooth locus of $X/S$ and outside the image of $\sigma$, so by the universal property of blow-ups (\textit{cf.} \cite[\href{https://stacks.math.columbia.edu/tag/0806}{Tag 0806}]{stacks-project}), we only need to describe $\phi$ above the \'etale localizations $\spec\O_{X,x}^{\,\on{\acute{e}t}}$, where $x,s,(C,D)$ are as in the statement of the lemma. We can assume that $S=\spec R$ is strictly local, with closed point $s$. Lift $T$ and $T'$ to global sections $\Delta,\Delta'$ of $S$, and pick an isomorphism
\[
\widehat{\O_{X,x}^{\,\on{\acute{e}t}}}=\widehat{R}[[u,v]]/(uv-\Delta\Delta')
\]
such that $C$ is locally given by $u=0$. The map
\[
\widehat{\sigma}\colon\widehat{\O_{X,x}^{\,\on{\acute{e}t}}}\ra\widehat{R}
\]
induced by $\sigma$ sends $u$ to a generator of $\Delta \widehat R$ and $v$ to a generator of $\Delta' \widehat R$. Scaling $u$ and $v$ by a unit of $\widehat R$ if necessary, we can assume $\widehat{\sigma}(u)=~\Delta$ and $\widehat{\sigma}(v)=~\Delta'$.

The completed local rings of $\spec \O_{X,x}^{\,\on{\acute{e}t}}\times_X X'$ can be computed using the blow-up of the algebra $B:=R[u,v]/(uv-\Delta\Delta')$ in the ideal $(u-\Delta,v-\Delta')$ (since the completion of $B$ at $(u,v,\m_R)$ is $\widehat{\O_{X,x}}$).
	
The ideal $(u-\Delta,v-\Delta')$  
is covered by two affine patches (with the obvious gluing maps):
	
	\begin{itemize} 
	
	\item The patch where $u-\Delta$ is a generator, given by the spectrum of
	\[
	R[u,v,\alpha]/((v-\Delta')-\alpha(u-\Delta),u\alpha+\Delta')\simeq R[u,\alpha]/(u\alpha+\Delta')
	\]
	since, in the ring $R[u,v,\alpha]/((v-\Delta')-\alpha(u-\Delta))$, the element $uv-\Delta\Delta'$ is equal to $(u-\Delta)(u\alpha+\Delta')$; 
	
	\item The patch where $v-\Delta'$ is a generator, where we obtain analogously the spectrum of
	\[
	R[v,\beta]/(v\beta+\Delta).
	\]
	\end{itemize}
	Thus we see that $X'$ remains nodal and that the edge $e$ of $\Gamma$ (of label $(\Delta\Delta')$) is replaced in $\Gamma'$ by a chain of two edges, one labelled $(\Delta)$ and one labelled $(\Delta')$. It also follows from this description that the strict transform of $C$ (resp.\ $D$) in $X'\times_X \spec\widehat{\O_{X,x}^{\,\on{\acute{e}t}}}$ contains the singularity of label $(\Delta')$ (resp.\ $(\Delta)$).
\end{proof}

\begin{cor}\label{corollary:unicite_zariski_locale_des_T-refinements}
With the same hypotheses and notation as in Lemma~\ref{structure des (C,delta)-blowups}, for any two sections $\sigma,\sigma'$ of $X/S$, denote by $Y\to X$ and $Y'\to X$ the blow-ups in the respective ideal sheaves of $\sigma$ and $\sigma'$. Then, $Y$ and $Y'$ are canonically isomorphic above the same type locus $X_{\sigma,\sigma'}$ of $\sigma$ and $\sigma'$ in $X$. Conversely, if $x$ is in $X\backslash X_{\sigma,\sigma'}$, then $Y$ and $Y'$ are not isomorphic above $\O_{X,x}$.
\end{cor}

\begin{proof}
The "conversely" part is immediate from Lemma \ref{structure des (C,delta)-blowups}. Pick a point $s\to S$ and a singular point $x$ of $X_s$ such that $\sigma(s)=\sigma'(s)=x$ and $\sigma,\sigma'$ have the same type $T$ at $x$. It suffices to exhibit a Zariski neighbourhood $V$ of $x$ in $X$ and an isomorphism $Y\times_X V \to Y'\times_X V$ compatible with the canonical isomorphisms $Y\times_X X^{\mathrm{smooth}}=X^{\mathrm{smooth}}=Y'\times_X X^{\mathrm{smooth}}$. Since $X,Y,Y'$ are finitely presented over $S$, this can be done assuming $S=\spec R$ is strictly local, with closed point $s$. Using the universal property of blow-ups (\textit{cf.} \cite[\href{https://stacks.math.columbia.edu/tag/0806}{Tag 0806}]{stacks-project}), we reduce to proving that the pull-back of the ideal sheaf of $\sigma'$ (resp.\ $\sigma$) to $Y$ (resp.\ $Y'$) is Cartier. The proofs are analogous, so we will only show that the pull-back to $Y$ of the ideal sheaf of $\sigma'$ is Cartier. This, in turn, reduces to proving that the ideal sheaf of $\sigma'$ in $\spec\widehat{\O_{X,x}^{\,\on{\acute{e}t}}}$ becomes Cartier in $Y\times_X \spec\widehat{\O_{X,x}^{\,\on{\acute{e}t}}}$. Pick an isomorphism
\[
\widehat{A}:=\widehat{R}[[u,v]]/(uv-\Delta_x)\simeq\widehat{\O_{X,x}^{\,\on{\acute{e}t}}}\,,
\]
where $\Delta_x\in R$ is a lift of the thickness of $x$. The map
\[
\widehat A \to \widehat{R}
\]
corresponding to $\sigma$ sends $u,v$ to elements $\Delta,\Delta'$ of $\widehat{R}$ with $\Delta\Delta'=\Delta_x$. Since $\sigma$ and $\sigma'$ have the same type at~$x$, there is a unit $\lambda\in\widehat{R}^\times$ such that the map
\[
\widehat A \to \widehat{R}
\]
corresponding to $\sigma'$ sends $u$ and $v$ to $\lambda\Delta$ and $\lambda^{-1}\Delta'$, respectively. We have reduced to proving that the sheaf given by the ideal $(u-\lambda\Delta, v-\lambda^{-1}\Delta')$ of $\widehat A$ becomes Cartier in the blow-up of $\widehat{A}$ in $(u-\Delta,v-\Delta')$. Put
\[
A=\widehat{R}[u,v]/(uv-\Delta\Delta');
\]
then it is enough to prove that the ideal $I=(u-\lambda\Delta, v-\lambda^{-1}\Delta')$ of $A$ becomes invertible in the two affine patches (as described in the proof of Lemma \ref{structure des (C,delta)-blowups}) forming the blow-up of $A$ in $(u-\Delta,v-\Delta')$. By analogy, we only check it in the patch where $u-\Delta$ is a generator, which is the spectrum of
\[
A_1=\widehat{R}[u,\alpha]/(u\alpha+\Delta'),
\]
where $v$ maps to $\Delta'+\alpha(u-\Delta)$. We have $I=(u-\lambda\Delta,\lambda v-\Delta')$, and in $A_1$ we can write
\begin{align*}
\lambda v -\Delta' & = \lambda (\Delta'+\alpha(u-\Delta)) +u\alpha \\
& = -\lambda\alpha\Delta+u\alpha \\
& = \alpha(u-\lambda\Delta).
\end{align*}
Thus, the preimage of $I$ in $A_1$ is the invertible ideal $(u-\lambda\Delta)$, and we are done.
\end{proof}

\begin{defi}\label{definition refinement}
Let $S$ be a smooth-factorial scheme and $X/S$ a quasisplit nodal curve, smooth over a dense open subscheme $U$ of $S$. We call \emph{basic refinement} of $X/S$ any morphism $f \colon X' \to X$ isomorphic to the blow-up of $X$ in the ideal sheaf of a section $\sigma\colon S \to X$. If $X/S$ is orientable at a point $x$ above which $f$ is not an isomorphism, it follows from Corollary \ref{corollary:unicite_zariski_locale_des_T-refinements} that the type $T$ of $\sigma$ at $x$ relative to an orientation $(C,D)$ is independent of the choice of $\sigma$: we say that $T$ is the \emph{type of $X' \to X$ at $x$}, or that $X' \to X$ is a \emph{basic $T\!$-refinement} (at $x$, relative to $(C,D)$).

We call \emph{refinement} of $X/S$ any morphism $f \colon X' \to X$ which, Zariski locally on $S$, is a composition of basic refinements.
\end{defi}

\begin{rem}\label{remarque les raffinements asymetriques existent et-localement et base change compatibles}
  \leavevmode
\begin{itemize}

\item If $S$ is excellent, then any geometric point $s\in S$ has an admissible neighbourhood $V$ by Proposition \ref{proposition:admissibles_are_a_basis}, so $X_V/V$ has a basic $T\!$-refinement for any type $T$ at any singular point of $X_s$.

\item Consider any morphism $S' \to S$, where $S'$ is still smooth-factorial (\textit{e.g.}\ any smooth map $S' \to S$). Let $x$ be a singular point of $X$ and $x'$ a singular point of $X'$ of image $x$. Then any type $T$ at $x$ pulls back to a type $T'$ at $x'$, and the base change to $S'$ of a basic refinement of type $T$ at $x$ is a basic refinement of type $T'$ at $x'$.

\item Let $f \colon X' \to X$ be a basic refinement, let $x \in X$ be a singular point at which $X/S$ is orientable and above which $f$ is not an isomorphism, and let $y$ be a generization of $x$. Let $T$ be the type of $f$ at~$x$. By Lemma \ref{lemma:type_passes_to_generization}, either $T$ corresponds to a type (still denoted by $T$) at $y$, in which case $X'\to X$ has type $T$ at $y$, or $T$ becomes trivial at $y$, in which case $f$ restricts to an isomorphism above a Zariski neighbourhood of $y$.

\end{itemize}
\end{rem}

\begin{prop}\label{proposition:resolutions}
Let $S$ be a smooth-factorial and excellent scheme and $U\subset S$ a dense open subscheme. Let $X/S$ be a nodal curve, smooth over $U$. Suppose that $S$ is an admissible neighbourhood of some point $s\in S$. Then, there exists a refinement $X' \to X$ such that all the singularities of $X'_s$ have prime thicknesses.
\end{prop}

\begin{proof}
By the definition of admissibility, $S$ remains an admissible neighbourhood of $s$ if we replace $X$ with a basic refinement. If $X' \to X$ is a basic refinement which is not an isomorphism above $s$, then by Lemma~\ref{structure des (C,delta)-blowups}, the arithmetic complexity of $X'$ at $s$ is strictly lower than that of $X$, so we obtain the proposition by induction.
\end{proof}

\begin{lem}\label{lemma:curves_with_prime_labels_are_loc_factorial}
Let $S$ be a smooth-factorial and excellent scheme and $U\subset S$ a dense open subscheme. Let $X/S$ be a quasisplit nodal curve, smooth over $U$. Let $x$ be a singular point of $x$ with prime thickness. Then $X$ is locally factorial at $x$. In particular, if $s\in S$ is such that all the singular points of $X_s$ have prime thicknesses, then $X\times_S \spec \O_{S,s}$ is locally factorial.
\end{lem}

\begin{proof}
Let $t$ be the image of $x$ in $S$. Then $\widehat{\O_{S,t}}$ is a unique factorization domain by Lemma~\ref{lemma:Popescu}. Thus, $\widehat{\O_{X,x}^{\,\on{\acute{e}t}}}$ is a unique factorization domain by Lemmas~\ref{lemme les irreductibles au complete sont les irreductibles} and~\ref{les anneaux locaux sont UFD}. Therefore, $\O_{X,x}$ itself is a unique factorization domain by Lemma \ref{lemma:factoriality_descends_under_faithfully_flat_ring_ext}.
\end{proof}

\section{N\'eron models of Jacobians}\label{sec6}

If $X\to S$ is a morphism of schemes, its relative Picard functor is the fppf sheafification of the functor sending an $S$-scheme $T$ to the group of isomorphism classes of line bundles on $X_T$. When $X/S$ is a nodal curve, by~\cite[Theorems~8.3.1 and~9.4.1]{NeronModels}, the Picard functor is representable by a smooth, quasiseparated $S$-group algebraic space $\pic_{X/S}$, the \emph{Picard space}. We write $\pic^{\on{tot}0}_{X/S}$ for the kernel of the degree map from $\pic_{X/S}$ to the constant sheaf $\Z$ on $S$ and $\pic^{0}_{X/S}$ for the fibrewise-connected component of identity of $\pic_{X/S}$, parametrizing line bundles of degree $0$ on every irreducible component of every fibre.

A classical way of obtaining a N\'eron model for the Jacobian $J$ of a proper smooth curve $X_U/U$ with a nodal model $X/S$, when $X$ is "nice enough", is to consider the quotient $P/E$, where $P=\pic^{\on{tot}0}_{X/S}$ and $E$ is the closure of the unit section in $P$, so that $P/E$ is the biggest separated quotient of $P$ (see, for example,~\cite[Section~9.5]{NeronModels}). This works well when three conditions are met: $P$ is representable by an $S$-algebraic space, $E$ is flat over $S$ (so that the quotient is also representable), and $\pic^{\on{tot}0}_{X/S}$ satisfies existence in the N\'eron mapping property (\textit{e.g.}\ $X$ is regular). However, this approach fails most of the time when $S$ is of arbitrary dimension since $E$ is rarely $S$-flat (\textit{cf.} \cite[Theorem 5.17]{Holmes}). The reason is that this method is designed to produce \emph{separated} N\'eron models, and most N\'eron models over higher-dimensional bases turn out to be non-separated.

In this section, we will work assuming $S$ is a smooth-factorial scheme and $U\subset S$ a dense open subscheme. In view of Corollary \ref{corollary:quotient_of_group_space_by_E^et_has_uniqueness_in_NMP}, it is tempting to try to construct the N\'eron model as the quotient of $P$ by the \'etale locus of $E/S$. This works when $P$ has existence in the N\'eron mapping property, \textit{i.e.}\ when $X$ is parafactorial along $X_U$ after any smooth base change (\textit{e.g.}\ regular). However, even if $X_U$ has nodal models, it may be that none of them remains parafactorial after every smooth base change. We will construct a N\'eron model $N$ for~$J$ when $X/S$ is arbitrary and give a local description of $N$ in terms of Picard spaces of local nodal models of~$X_U$. Then, we will give a simple combinatorial criterion for $N$ to be separated, related to the alignment condition of~\cite{Holmes}.

\subsection{Construction of the N\'eron model}

\begin{rem}
Let $S$ be a smooth-factorial and excellent scheme, $U\subset S$ a dense open subscheme and $X/S$ a quasisplit nodal curve, smooth over $U$. Suppose that every singular point of $X/S$ has prime thickness. Then $X/S$ is locally factorial by Lemma \ref{lemma:curves_with_prime_labels_are_loc_factorial}, so any $U$-point of $P:=\pic^{\on{tot}0}_{X/S}$ extends to an $S$-section. By Corollary \ref{corollary:irred_of_etale_loc_ring_stable_under_smooth_maps}, this remains true after base change to any smooth $S$-scheme, so $P$ satisfies existence in the N\'eron mapping property. Thus, by Corollary \ref{corollary:quotient_of_group_space_by_E^et_has_uniqueness_in_NMP}, the quotient of $P$ by the \'etale locus of the closure of its unit section is the N\'eron model of the Jacobian of $X_U$. However, we cannot always use Lemma \ref{proposition:admissibles_are_a_basis} and Proposition \ref{proposition:resolutions} to reduce locally to this situation since some singular points of $X/S$ with prime thickness may have singular generizations whose thickness is not prime (see Example \ref{example:irred_not_etale_irred}).
\end{rem}

\begin{lem}\label{proposition:refinements_induce_open_immersions_of_Pics}
Let $S$ be a smooth-factorial excellent scheme and $U\subset S$ a dense open subscheme. Let $X/S$ be a nodal curve smooth over $U$, and let $f$ be a refinement $X' \to X$. Write $P=\pic^{\on{tot}0}_{X/S}$ and $P'=\pic^{\on{tot}0}_{X'/S}$. Denote by $E$ $($resp.\ $E')$ the scheme-theoretic closure of the unit section in $P$ $($resp.\ $P')$. Then, the canonical morphisms $P \to P'$, $P/E \to P'/E'$ and $P/E^{\mathrm{\acute{e}tale}}\to P'/E'^{\,\mathrm{\acute{e}tale}}$ are open immersions. In addition, if $X/S$ is quasisplit and $f$ is an isomorphism above every singular point of $X/S$ which is not disconnecting in its fibre, then $P/E \to P'/E'$ is an isomorphism.
\end{lem}

\begin{proof}
When $X/S$ is quasisplit, $\sing(X/S)$ is the disjoint union of its open and closed subschemes consisting of points that are, respectively, disconnecting in their fibre and non-disconnecting in their fibre. This partition is compatible with base change and refinements, by Lemma \ref{structure des (C,delta)-blowups} and Proposition \ref{graphes duaux et changement de base}. By Corollary~\ref{corollary:curves_are_QS_over_some_etale_cover}, we may assume that $X/S$ and $X'/S$ are quasisplit. By induction, we may assume that $f$ is a basic refinement. By the proof of Lemma \ref{structure des (C,delta)-blowups}, we may therefore assume that there exists a closed subscheme $F$ of $\sing(X/S)$ such that $f$ is an isomorphism above $X\backslash F$ and such that $X'\times_X F\simeq\P^1_F$. Hence, the pull-back along $f$ induces an equivalence of categories between line bundles on $X$ and line bundles on $X'$ having degree $0$ on every irreducible component of every fibre of $X'\times_X F \to F$. In particular, we have a canonical isomorphism $\pic^0_{X/S}=\pic^0_{X'/S}$. As $\pic^0_{X/S}$ is an open neighbourhood of the unit section in $P$ (and similarly for $P'$), it follows that $P \to P'$ is a local isomorphism. Since it is also injective, it is an open immersion. It follows that all squares are cartesian in the commutative diagram
\[
\xymatrix{
E^{\mathrm{\acute{e}tale}} \ar[r]\ar[d] & E'^{\mathrm{\acute{e}tale}} \ar[d] \\
E \ar[r]\ar[d] & E' \ar[d] \\
P \ar[r] & P'\rlap{.}
}
\]
Therefore, $P/E \to P'/E'$ and $P/E^{\mathrm{\acute{e}tale}} \to P'/E'^{\mathrm{\acute{e}tale}}$ are open immersions as well.

We now prove that $P/E \to P'/E'$ is surjective, assuming that every point of $F$ is disconnecting in its fibre over~$S$. This may be checked at the level of \'etale stalks over $S$: it suffices to show that $P(S)$ surjects onto $P'/E'(S)$, assuming that $S=\spec R$ is strictly local with closed point $s$. If $F$ is empty, we are done. Otherwise, $F_s$ is a disconnecting singular point $x$ of $X_s$, and a line bundle on $X'$ is in the image of $P(S)$ if and only if its restriction to $X'\times_X F$ is trivial, \textit{i.e.}\ if and only if it has degree $0$ on $X'\times_X x\simeq\mathbb P^1_x$. Therefore, it suffices to show that $E'$ contains a line bundle $\mathcal L$ of degree $1$ on $X'\times_X x\simeq\mathbb P^1_x$. Let $x'$ be a singular point of $X'_x$, and let $\Delta\in R$ be a lift of its thickness. Let $E$ be the connected component of $\sing(X'/S)$ containing~$x'$. The map $E \to S$ is a closed immersion cut out by $\Delta$, and $\left(X'\times_S E)\right)\backslash E$ has two connected components. Denote by $C$ the one whose fibre over $x$ is nonempty and $D$ the other one. The scheme-theoretic closures of $C$ and $D$ in $X'$ are effective Cartier divisors, which we still call $C$ and $D$ (to see that they are Cartier at $x'$, notice that they coincide locally with an orientation at $x'$ as in Definition \ref{definition:orientation}). Let $\mathcal{L}$ be the line bundle corresponding to $D$. Then $\mathcal{L}$ has degree $1$ on $X'\times_X x$ since $C$ and $D$ meet transversally at $x'$, and the $S$-point of $P'$ corresponding to $\mathcal{L}$ is in $E'$ since $\mathcal{L}$ is trivial over $U$, so we are done.
\end{proof}

\begin{lem}\label{lemma:line_bundles_on_loc_rings_extend_locally}
Let $X \to S$ be a proper and finitely presented morphism of schemes, $s$ be a point of $S$ and $\mathcal{L}$ be a line bundle on $X^s:=X \times_S \spec(\O_{S,s})$. Then there exists a Zariski open neighbourhood $S'$ of $s$ in $S$ such that $\mathcal{L}$ extends to a line bundle on $X\times_S S'$.
\end{lem}

\begin{proof}
Pick a Cartier divisor $D=\{(U_i,f_i)\}_{i\in I}$ representing $\mathcal{L}^t$. Since $X^s$ is quasicompact, we may assume that the index set $I$ is finite. For each $i\in I$, pick an open subscheme $V^i$ of $X$ containing $U_i$ and a $V_i$-section $g_i$ of $\mathcal{K}_X:=\Frac\O_X$ restricting to $f_i$. Shrinking $V_i$ if necessary, we may assume that $g_i/g_j$ is in $\O_X^\times(V_i\cap V_j)$. The union of the $V_i$ is an open subset $V$ of $X$ containing $X_t$. Therefore, since $X/S$ is proper, the image of $X\backslash V$ in $S$ is a closed subset not containing $t$, and its complement $S'$ is an open neighbourhood of $t$ in $S$. The $V_i$ cover $X_{S'}$, so $\{(V_i,g_i)\}$ is a Cartier divisor on $X_{S'}$ which restricts to $D$, and the corresponding line bundle extends $\mathcal L$.
\end{proof}

\begin{lem}\label{lemma:equivalence_classes_loc_constructible}
Let $S$ be a smooth-factorial and excellent scheme, $U\subset S$ a dense open subscheme and $X/S$ a nodal curve, smooth over $U$. For any geometric point $\bar s \to S$, write $J_{\bar s}$ for the set of prime factors of thicknesses of singular points of $X_{\bar s}$ and $\overline M_{\bar s}$ for the submonoid of $\overline{O_{S,\bar s}^{\on{\acute{e}t}}}$ spanned by $J_{\bar s}$. Consider the relation $R$ on $S$ given by $sRt$ whenever $t$ specializes $s$ and for some $($equivalently, any$)$ \'etale specialization $\bar t$ of $\bar s$, where $\bar t, \bar s$ are geometric points above $s$ and $t$, the restriction map of \'etale stalks induces a canonical isomorphism between $J_{\bar t}$ and $J_{\bar s}$. Denote by $\sim_S$ the transitive closure of $R$. Then, the equivalence classes for $\sim_S$ are locally constructible subsets of $S$.
\end{lem}

\begin{proof}
We immediately reduce to the following claim: given a point $s \in S$, if we denote by $C_s$ its equivalence class for $\sim_S$, then the intersection of $C_s$ with a small enough Zariski neighbourhood of $s$ in $S$ is locally constructible in $S$. We will now prove the claim. For any \'etale neighbourhood $(V,v)$ of $s$, the preimage of $C_s$ in $X_V$ has a locally finite number of connected components, all of which are classes for $\sim_{V}$ (where $\sim_{V}$ is defined as $\sim_S$, but after replacing $X/S$ by $X_V/V$). Therefore, there exists a Zariski neighbourhood $(W,w)$ of $v$ in $V$ such that the preimage of $C_s$ in $W$ is the class $C_w$ of $w$ for $\sim_W$. It follows that the image of $C_w$ in $S$ is the intersection of $C_s$ with an open of $S$ (namely, the image of $W$). Thus, the claim may be proved after replacing $(X/S,s)$ by $(X_W/W,w)$. In particular, we may assume that $S$ is an admissible neighbourhood of $s$ (since admissibility is preserved by \'etale localization).

Pick a geometric point $\bar s$ above $s$. The singular locus of $X/S$ has finitely many connected components $(F_1,\ldots,F_r)$, and each $F_i \to S$ is cut out by a global section $a_i$ of $\O_S$. Since $S$ is an admissible neighbourhood of $s$, the elements of $J_{\bar s}$ lift to global sections $\Delta_1,\ldots,\Delta_n$ of $\O_S$. Shrinking $S$ further if necessary, we may assume that every $\Delta_i$ divides some $a_j$ in $\O_S(S)$ (and not just in $\O_{S,s}^{\,\on{\acute{e}t}}$). Recall that an integral scheme $Y$ is called \emph{geometrically unibranch} at a point $y\in Y$ if the strict henselization of $\O_{Y,y}$ is integral or, equivalently by~\cite[corollaire~IX.1]{Raynaud}, if the integral closure of $\O_{Y,y}$ is local with purely inseparable residue extension over $\O_{Y,y}$. Denote by $Z_i$ the closed subscheme of $S$ cut out by $\Delta_i$. The set of points $Z_i^{\on{uni}}$ at which $Z_i$ is geometrically unibranch is locally constructible in $Z_i$ by~\cite[corollaire 9.7.10]{EGA4.3}. Therefore, the intersection in $S$ of the images of the $Z_i^{\on{uni}}$ for all $i$ is locally constructible in $S$. This intersection is precisely $C_s$, so we are done.
\end{proof}

\begin{rem}
With the hypotheses and notation of Lemma \ref{lemma:equivalence_classes_loc_constructible}, the equivalence classes for $\sim_S$ form a partition of $S$ into locally constructible subsets. In particular, locally on $S$, there are only finitely many such classes. Since $\sim_S$ only depends on $X$ via the sets $J_{\bar s}$, it remains unchanged if we replace $X$ with a refinement.
\end{rem}

\begin{thm}\label{theorem:NMs_Jac}
Let $S$ be a smooth-factorial and excellent scheme, and let $U\subset S$ be a dense open subscheme. Let $X_U/U$ be a smooth curve that admits a nodal model over $S$. Then:

\begin{enumerate}

\item \label{thm7.6-1} The Jacobian $J=\pic^0_{X_U/U}$ of $\,X_U/U$ admits a N\'eron model $N$ over $S$.

\item \label{thm7.6-2} For any nodal model $X/S$ of $\,X_U/U$, the map $\pic^{\on{tot}0}_{X/S}/E^{\mathrm{\acute{e}tale}} \to N$ extending the identity over $U$ is an open immersion, where $E$ is the scheme-theoretic closure of the unit section in $\pic^{\on{tot}0}_{X/S}$.

\item \label{thm7.6-3} For any \'etale morphism $V \to S$ and nodal $V$-model $X$ of $X_{U\times_S V}$, if $\bar s \to V$ is a geometric point such that the singularities of $X_{\bar s}$ have prime thicknesses, then the canonical map $\pic^{\on{tot}0}_{X/V} \to N$ is surjective on $\spec(\O_{S,\bar s}^{\,\on{\acute{e}t}})$-points.
\end{enumerate}

\end{thm}

\begin{rem}\label{remark:N_covered_by_some_Pic/E_after_etale_cover}
In the setting of Theorem \ref{theorem:NMs_Jac}, if $X^0$ is any nodal model of $X_U$, by Propositions~\ref{proposition:admissibles_are_a_basis} and~\ref{proposition:resolutions}, there exist an \'etale cover $V \to S$ and a refinement $X \to X^0_V$ with the following property: for any $s\in S$, there is some geometric point $v \to V$ above $s$ such that the singularities of $X_v$ have prime thicknesses. In particular, it follows from Theorem \ref{theorem:NMs_Jac} that the canonical map $\pic^{\on{tot}0}_{X/V} \to N$ is an \'etale cover.
\end{rem}

\begin{proof}[Proof of Theorem \ref{theorem:NMs_Jac}]
  Recall that the formation of N\'eron models is smooth local on the base (\textit{cf.} Propositions~\ref{changement de base lisse} and~\ref{proposition descente lisse des NM}), that the properties of morphisms "being \'etale" and "being an open immersion" are fpqc local on the target (\textit{cf.} \cite[\href{https://stacks.math.columbia.edu/tag/02L3}{Tags 02L3} and \href{https://stacks.math.columbia.edu/tag/02VN}{02VN}]{stacks-project}) and that for a nodal curve $X/S$, the formation of $\pic^{\on{tot}0}_{X/S}$ and of the closure of its unit section commute with flat base change. In particular, claims~\eqref{thm7.6-1}, \eqref{thm7.6-2} and \eqref{thm7.6-3} of the theorem hold if and only if they hold \'etale-locally on $S$.

First, let us assume \eqref{thm7.6-1} and~\eqref{thm7.6-2} and prove \eqref{thm7.6-3}. Let $X,V, \bar s$ be as in \eqref{thm7.6-3}, and put $P=\pic^{\on{tot}0}_{X/V}$ and $T=\spec\O_{S,\bar s}^{\,\on{\acute{e}t}}$. It follows from \eqref{thm7.6-2} that $P \to N$ is \'etale, and we only need to show that it is surjective on $T$-points. Pick $f_U\in N(T)$. Then, $f_U$ corresponds to a line bundle $\mathcal{L}_U$ on $X_{T_U}$. The curve $X_T$ is locally factorial by Lemma~\ref{lemma:curves_with_prime_labels_are_loc_factorial}. Pick a Weil divisor $D$ on $X_U$ representing $\mathcal{L}_U$. Its closure $\overline D$ in $X_T$ is Cartier by local factoriality, hence defines a line bundle $\mathcal{L}$ extending $\mathcal{L}_U$, \textit{i.e.}\ a $T$-point of $P$ mapping to $f_U$.

Now, it suffices to prove that \eqref{thm7.6-1} and \eqref{thm7.6-2} hold. Fix a nodal $S$-model $X$ of $X_U$. We say that a smooth $S$-scheme $V$ is \emph{good} if the following two conditions are met:
\begin{itemize}
\item There exists a $V$-N\'eron model $N_V$ for $J_{U\times_S V}$. (The notation is unambiguous since if $X_U$ has a N\'eron model $N$, then $N\times_S V$ is the $V$-N\'eron model of $X_{U\times_S V}$.)
\item For any \'etale map $V' \to V$ and any refinement $X' \to X_{V'}$, the canonical map $\pic^{\on{tot}0}_{X'/V'}/E'^{\,\mathrm{\acute{e}tale}} \to N_{V'}$ is an open immersion,  where $E'$ is the closure of the unit section in $\pic^{\on{tot}0}_{X'/V'}$.
\end{itemize}
We say that $V$ has the property $\mathcal P$ if there exists a good open subscheme $W$ of $V$ such that $V$ is an admissible neighbourhood of every point of $V\backslash W$. If $s$ is a point of $S$, we say that $s$ is good (resp.\ has $\mathcal P$) if some \'etale neighbourhood of $s$ is good (resp.\ has $\mathcal P$). Goodness and $\mathcal{P}$ can both be checked locally on $S$ for the \'etale topology. Therefore, the theorem reduces to the following two claims: goodness and $\mathcal{P}$ are equivalent, and any point of $S$ has $\mathcal{P}$. We will now prove these claims, in order. Throughout the rest of the proof, when $V$ is a smooth $S$-scheme, we will write $\sim_V$ for the equivalence relation on $V$ defined as in Lemma \ref{lemma:equivalence_classes_loc_constructible} (relative to $X_V/V$ or, equivalently, to any refinement of it).

Clearly goodness implies $\mathcal{P}$. We will show that $S$ is good, assuming it has $\mathcal P$. Goodness and $\mathcal{P}$ are local, so it suffices to pick a point $s\in S$ and show that, after shrinking $S$ at will to an arbitrarily small \'etale neighbourhood of $s$, $J$ has a N\'eron model $N$ and for any refinement $X' \to X$, the canonical map $P/E^{\mathrm{\acute{e}tale}} \to N$ is an open immersion, where $P=\pic^{\on{tot}0}_{X'/S}$ and $E$ is the closure of the unit section in $P$.

Let $W$ be a good open subscheme of $S$ such that $S$ is an admissible neighbourhood of every point of $F:=S\backslash W$. Shrinking $S$, we may assume that there are only finitely many equivalence classes $(C_i)_{i\in I}$ for $\sim_S$. By Proposition \ref{proposition:refinements_induce_open_immersions_of_Pics}, we may replace $X'$ with a refinement. Pick an index $j\in I$ such that $C_j$ meets $F$, and let $s'$ be a point of $F\cap C_j$. By Proposition \ref{proposition:resolutions}, we may assume that the singular points of $X'$ above $s'$ have prime thicknesses, which implies that the singular points of $X'$ mapping to $C_j$ have prime thicknesses by the definition of $\sim_S$. Iterating the process, we may assume that the thicknesses of all the singularities of $X'$ above $F$ are prime. Denote by $N_W$ the $W$-N\'eron model of $J_{U\times_S W}$. The canonical map
\[
(P/E^{\mathrm{\acute{e}tale}})\times_S W \to N_W
\]
is an open immersion since $W$ is good. Denote by $N$ the gluing of $P/E^{\mathrm{\acute{e}tale}}$ and $N_W$ along $(P/E^{\mathrm{\acute{e}tale}})\times_S W$ (the notation $N_W$ is unambiguous since $N\times_S W=N_W$). Then $N$ is a smooth $S$-model of $J$ with uniqueness in the N\'eron mapping property and with an open immersion $P/E^{\mathrm{\acute{e}tale}} \to N$ restricting to the identity over~$U$. Therefore, in order to prove that $S$ is good, it suffices to show that for any smooth $S$-scheme $Y$, the restriction map
\[
\Hom_S(Y,N) \to \Hom_{U}(Y_U,N_U)
\]
is surjective. Pick some $f_U \in \Hom_U(Y_U,N_U)$. By uniqueness in the mapping property, it suffices to show that $f_U$ extends to an $S$-map $Y' \to N$ for a Zariski neighbourhood $Y'$ of a given point $y\in Y$. If $y$ is in~$Y_W$, this is clear. Otherwise, $y$ lands in~$F$ so the singularities of $X_Y$ above $y$ have prime thicknesses by Corollary~\ref{corollary:irred_of_etale_loc_ring_stable_under_smooth_maps}. In particular, the base change $X^y$ of $X_Y$ to $\spec(\O_{Y,y})$ is locally factorial by Lemma \ref{lemma:curves_with_prime_labels_are_loc_factorial}, so the line bundle on $X^y \times_S U$ corresponding to $f_U$ extends to a line bundle $\mathcal L$ on $X^y$. Then, Lemma \ref{lemma:line_bundles_on_loc_rings_extend_locally} provides an open neighbourhood $Y'$ of $y$ in $Y$ and a line bundle on $X_{Y'}$ extending $\mathcal L$, \textit{i.e.}\ a morphism $Y' \to P$ extending $f_U$. We conclude by composing with $P \to N$.

We have shown that goodness and $\mathcal P$ are equivalent. Now, let $s$ be a point of $S$; we will prove that $s$ has~$\mathcal{P}$. For any \'etale morphism $V \to S$, the locally closed pieces of the equivalence classes of $\sim_V$ form a partition of $V$ into locally closed subsets. We write $n(V)$ for the number of pieces of this partition and $n(s)$ for the minimum of the $n(V)$ where $V$ ranges through the \'etale neighbourhoods of $s$ in $S$. By the local constructibility of the classes for $\sim_S$, we know that $n:=n(s)$ is finite. We will show that $s$ has $\mathcal P$ by induction on $n$. Shrinking $S$, we may assume that $n(S)=n$. If $n=1$, then $X/S$ is smooth, so $S$ is good and we are done. Otherwise, denote by $F_0$ the equivalence class of $s$ for $\sim_S$, and let $F$ be the locally closed piece of~$F_0$ containing $s$. By the minimality of $n(S)$, $F$ is closed in $S$. Therefore, the open subscheme $W=S\backslash F$ of~$S$ is such that $n(W)=n(S)-1$, and by induction $W$ has $\mathcal{P}$; \textit{i.e.}\ $W$ is good. Let $(V,v)$ be an admissible neighbourhood of $s$ in $S$. Shrinking $V$, we may assume that all the points of $V \backslash W_V$ are $\sim_V$-equivalent, from which it follows that $V$ is an admissible neighbourhood of all of them. Then $V$ has $\mathcal P$, which concludes the proof.
\end{proof}

\begin{rem}\label{remark:strict_log_jac_vs_pic_aggregate}
In~\cite{HMOPModelsJacobians}, the authors describe the \emph{strict logarithmic Jacobian} of a logarithmic curve. They show that when $X/S$ is a nodal curve over a toroidal variety, smooth over the complement $U$ of the boundary divisor,  there are canonical log structures on $X$ and $S$ such that the strict logarithmic Jacobian is the N\'eron model of $X_U$. This gives a moduli interpretation in logarithmic geometry for the N\'eron model constructed in Theorem \ref{theorem:NMs_Jac} when $U$ is the complement in $S$ of a divisor with normal crossings. A similar interpretation can be given when the discriminant locus of $X/S$ is arbitrary. Indeed, let $M_S$ be the \'etale subsheaf of monoids of $\O_S$ whose \'etale stalks are generated by the units and by the prime factors of the singular ideals of $X$. Let $M_X$ be the submonoid of $\O_X$ whose \'etale stalk at a geometric point $x\to X$ above a geometric point $s\to S$ is:
\begin{itemize}
\item The submonoid of $\O_{X,x}^{\,\on{\acute{e}t}}$ spanned by $(\O_{X,x}^{\,\on{\acute{e}t}})^\times$ and $M_{S,s}$ if $x$ is smooth over $S$; 
\item The submonoid of $\O_{X,x}^{\,\on{\acute{e}t}}$ spanned by $(\O_{X,x}^{\,\on{\acute{e}t}})^\times$, $M_{S,s}$ and local parameters for the two branches of $X/S$ at $x$ if $x$ is singular.
\end{itemize}
Then, $M_S\to \O_S$ and $M_X\to \O_X$ are logarithmic structures in the sense of~\cite{Kato}, but they do not necessarily admit \'etale-local charts (\textit{cf.} Example \ref{example:irred_not_etale_irred}). Therefore, $(X,M_X)$ and $(S,M_S)$ are not logarithmic schemes in the usual sense, and we cannot apply directly the results of~\cite{HMOPModelsJacobians} (although many of the arguments remain valid in our context). However, replacing $U$ with the maximal open subscheme of $S$ over which $X$ is smooth, \textit{i.e.}\ on which $M_S=\O_S^\times$, we find that the groupification $M_X^{\mathrm{gp}}$ coincides with the direct image of $\O_{X_U}^\times$ in $X$. Hence, two isomorphism classes of $M_X^{\mathrm{gp}}$-torsors which coincide over $X_U$ are equal. Write $H^1(X,M_X^{\mathrm{gp}})^\dagger$ for the subgroup of $H^1(X,M_X^{\mathrm{gp}})$ consisting of torsors which, locally on $S$, come from a line bundle on a refinement of $X$. Then it follows that $H^1(X,M_X^{\mathrm{gp}})^\dagger=\Hom(U,\pic^0_{X_U/U})$. Combining this with the fact that the formation of $M_S$ and $M_X$ commutes with smooth base change by Corollary \ref{corollary:irred_of_etale_loc_ring_stable_under_smooth_maps}, we find that
\begin{align*}
(\sch/S)^{\on{op}} & \to \set \\
T & \mapsto H^1(X_T,M_{X_T}^{\mathrm{gp}})^\dagger
\end{align*}
is the Hom functor of the N\'eron model of $\pic^0_{X_U/U}$. As in~\cite{MolchoWise} or~\cite{HMOPModelsJacobians}, we can describe explicitly $H^1(X,M_X^{\mathrm{gp}})^\dagger$ as the subgroup of $H^1(X,M_X^{\mathrm{gp}})$ consisting of torsors satisfying a certain condition that can be expressed in terms of dual graphs, the condition of \emph{bounded monodromy}.
\end{rem}

\subsection{A criterion for separatedness}

In this subsection, we exhibit a necessary and sufficient combinatorial condition for the N\'eron model of Theorem \ref{theorem:NMs_Jac} to be separated, closely related to the alignment condition of~\cite{Holmes}.

\begin{defi}\label{c-strict alignment}
Let $\Gamma$ be a graph labelled by a monoid $M$, written multiplicatively. Following~\cite[Definition 2.11]{Holmes}, we say that $\Gamma$ is \textit{aligned} when for every cycle $\Gamma^0$ in $\Gamma$, all the labels figuring in $\Gamma^0$ are positive powers of the same element $l$ of $M$. When $S$ is a smooth-factorial scheme and $s\to S$ a geometric point, we say that a nodal curve $X/S$ is \textit{aligned at~$s$} when its dual graph $\Gamma_s$ at $s$ is aligned. We say $X/S$ is \emph{aligned} if it is aligned at every geometric point of $S$.

	If $M$ is the free commutative monoid over a set of generators $G$, we say that $\Gamma$ is \textit{strictly aligned} if $l$ can be chosen in $G$. We say that $X/S$ is \textit{strictly aligned at $s$} if $\Gamma_s$ is strictly aligned (here $G$ is the set of principal prime ideals of $\O^{\on{\acute{e}t}}_{S,s}$). We say that $X/S$ is \emph{strictly aligned} if it is strictly aligned at every geometric point of $S$.
\end{defi}

\begin{ex}
Over $S=\spec\C[[u,v]]$, at the closed point, among the  following three dual graphs, the first is not aligned, the second is aligned but not strictly, and the third is strictly aligned.

\begin{center}
\begin{tikzpicture}
  \SetGraphUnit{4}
  \Vertex{B}
  \WE(B){A}
  \Edge[label = $(uv^ 2)$](A)(B)
  \Edge[label = $(u^ 2v^ 2)$](B)(A)
\end{tikzpicture}
\end{center}

\begin{center}
\begin{tikzpicture}
  \SetGraphUnit{4}
  \Vertex{A}
  \Loop[style={}, dist = 3cm, dir = EA, label = $(uv)$](A.north)
\end{tikzpicture}
\end{center}

\begin{center}
\begin{tikzpicture}
  \SetGraphUnit{4}
  \Vertex{B}
  \WE(B){A}
  \Edge[label = $(u^2)$](A)(B)
  \Edge[label = $(u^3)$](B)(A)
  \Loop[style={}, dist=3cm, dir=NO, label=$(v^3)$](A.west)
\end{tikzpicture}
\end{center}
\end{ex}

\begin{ex}
In the setting of Example \ref{example:irred_not_etale_irred}, the curve $X/S$ is strictly aligned at the closed point $s$ of $S$ (since $S$ is strictly local and the dual graph at $s$ is a loop with prime label), but $X/S$ is not aligned.
\end{ex}

\begin{prop}\label{alignement et platitude de E}
	Let $S$ be a regular scheme, $U\subset S$ a dense open and $X/S$ a nodal curve, smooth over $U$. Let $P=\pic^{\on{tot}0}_{X/S}$, and let $E$ be the scheme-theoretic closure in $P$ of its unit section. Then the following conditions are equivalent:
	
	\begin{enumerate}
		\item $E/S$ is flat.
		\item $E/S$ is \'etale.
		\item $X/S$ is aligned.
	\end{enumerate}
\end{prop}

\begin{proof}
This is~\cite[Theorem 5.17]{Holmes}.
\end{proof}

\begin{thm}\label{theorem:separatedness_of_NM_jac_iff_strictly_aligned}
Let $S$ be a regular and excellent scheme, $U\subset S$ a dense open subscheme and $X/S$ a nodal curve, smooth over $U$. Denote by $J$ the Jacobian of $X_U/U$. Then, the $S$-N\'eron model $N$ of $J$ exhibited in Theorem~\ref{theorem:NMs_Jac} is separated if and only if $X/S$ is strictly aligned.
\end{thm}

\begin{proof}
First, suppose that $N$ is separated. Let $s \to S$ be a geometric point; we will show that $X/S$ is aligned at~$s$. By Corollary \ref{corollaire le NM passe aux limites d'algebres lisses}, this may be checked assuming that $S$ is strictly local with closed point $s$. Proposition~\ref{proposition:resolutions} provides a refinement $X' \to X$ such that every singularity in the closed fibre of $X'/S$ has prime thickness. Let $\Gamma, \Gamma'$ be the dual graphs at $s$ of $X$ and $X'$, respectively. The closure of the unit section in $\pic^{\on{tot}0}_{X'/S}$ is \'etale over $S$ by~\cite[Theorem 6.2]{Holmes}. Hence, $\Gamma'$ is aligned at $s$ by Proposition \ref{alignement et platitude de E}, and even strictly aligned since its labels are prime. Since $\Gamma'$ refines $\Gamma$ (\textit{cf.} Lemma \ref{structure des (C,delta)-blowups}), it follows that $\Gamma$ is strictly aligned as well.

Conversely, suppose that $X$ is strictly aligned. We will show that $N\to S$ is separated. This may be done assuming that $S=\spec R$ is strictly local with closed point $s$. Replacing $X$ with a refinement, we may assume that every singular point of $X_s$ has prime thickness by Proposition \ref{proposition:resolutions}. Put $P=\pic^{\on{tot}0}_{X/S}$, and let~$E$ be the scheme-theoretic closure of the unit section in $P$. Then $E$ is \'etale over $S$ by Proposition \ref{alignement et platitude de E}, and there is a canonical open immersion $P/E \to N$ by Theorem \ref{theorem:NMs_Jac}. Since $P/E$ is separated, it suffices to show that this open immersion is surjective. This can be checked on \'etale stalks over $S$: let $t \to S$ be a geometric point and $T=\spec R'$, where $R'=\O_{S,t}^{\,\on{\acute{e}t}}$; it suffices to show that $P/E \to N$ is surjective on $T$-points. Proposition~\ref{proposition:resolutions} provides a refinement $X' \to X_T$ such that the singular points of $X'$ above $t$ have prime thicknesses. Put $P'=\pic^{\on{tot}0}_{X'/T}$. By part~\eqref{thm7.6-3} of Theorem \ref{theorem:NMs_Jac}, the map $P'(T) \to N(T)$ is surjective. Therefore, by the "in addition" part of Proposition \ref{proposition:refinements_induce_open_immersions_of_Pics}, it suffices to show that $X' \to X_T$ is an isomorphism above every singular point of $X_t$ which is non-disconnecting in $X_t$. Let $x$ be such a point. By Lemma \ref{structure des (C,delta)-blowups}, it is enough to prove that the thickness $T\in\overline{R'}$ of $x$ is prime. Let $y$ be the singular point of $X_s$ specializing $x$. Let $\Delta\in R$ be a lift of the thickness of $y$. Then $T$ is the principal ideal of $R'$ spanned by $\Delta$. The ring $R'/(\Delta)$ is reduced as a filtered colimit of \'etale $R/(\Delta)$-algebras, so $\Delta$ is square-free in $R'$. But $\Delta$ is a prime power in~$R'$ since $X$ is strictly aligned at $t$, so it is prime and we are done.
\end{proof}

\begin{rem}
As mentioned in the proof of~\cite[Theorem 5.17]{Holmes}, Proposition \ref{alignement et platitude de E} should still hold if we only require $S$ to be smooth-factorial instead of regular. If so, our proof of Theorem \ref{theorem:separatedness_of_NM_jac_iff_strictly_aligned} remains valid when~$S$ is just an excellent and smooth-factorial scheme.
\end{rem}


\newcommand{\etalchar}[1]{$^{#1}$}


\begin{thebibliography}{HMO{\etalchar{+}}20+++}

\bibitem[BLR90]{NeronModels}
S.~Bosch, W.~Lütkebohmert and M.~Raynaud,
{\em N\'eron Models},
Ergeb.\ Math.\ Grenzgeb.~(3), 21, 
  Springer-Verlag, Berlin, 1990.

\bibitem[Cap08]{Caporaso2008Neron-models-an}
L.~Caporaso, 
 {\em N{\'e}ron models and compactified {P}icard schemes over the moduli  stack of stable curves}, 
  Amer.\ J.\ Math. \textbf{130} (2008), no.~1, 1--47.

\bibitem[Dan70]{Danilov1970On-A-Conjecture}
V.\,I.~Danilov, 
  {\em On a conjecture of {Samuel}}, 
  Math.\ USSR Sb.\ \textbf{10} (1970), no.~1, 127.

\bibitem[EL19]{QuadChabauty}
B.~Edixhoven and G.~Lido, 
  {\em Geometric quadratic {Chabauty}}, preprint 
  \arXiv{1910.10752} (2019).

\bibitem[GD60]{EGA1}
A.~Grothendieck and J.~Dieudonn\'e, 
  {\em \'El\'ements de g\'eom\'etrie alg\'ebrique : I. Le langage des
  sch\'emas}, Publ.\ Math.\ Inst.\ Hautes \'Etudes Sci.\ \textbf{4} (1960).

\bibitem[GD64]{EGA4.1}
  \bysame,
  {\em \'El\'ements de g\'eom\'etrie alg\'ebrique : IV. \'Etude locale des
    sch\'emas et des morphismes de sch\'emas, Premi\`ere partie},
 Publ.\ Math.\ Inst.\ Hautes \'Etudes Sci.\ \textbf{20} (1964). 

\bibitem[GD66]{EGA4.3}
\bysame, 
{\em \'El\'ements de g\'eom\'etrie alg\'ebrique : IV. \'Etude locale des
  sch\'emas et des morphismes de sch\'emas, Troisi\`eme partie},
 Publ.\ Math.\ Inst.\ Hautes \'Etudes Sci.\ \textbf{28} (1966). 

\bibitem[Hol15]{HolmesUniversalJacobian}
D.~Holmes, 
  {\em A {N\'eron} model of the universal jacobian}, 
  preprint \arXiv{1412.2243} (2015).

  
\bibitem[Hol19]{Holmes}
\bysame, 
  {\em N\'eron models of {Jacobians} over base schemes of dimension greater
  than 1}, J.~reine angew.\ Math. \textbf{747} (2019), 109--145. 

\bibitem[Hol21]{HolmesExtendingDRC}
\bysame,  
  {\em Extending the double ramification cycle by resolving the
    Abel-Jacobi map}, J.~Inst.\ Math.\ Jussieu \textbf{20} (2021), no.~1, 331--359.
  
\bibitem[HMO{\etalchar{+}}20]{HMOPModelsJacobians}
D.~Holmes, S.~Molcho, G.~Orecchia and T.~Poiret,  
  {\em Models of {Jacobians} of curves}, preprint 
  \arXiv{2007.10792} (2020).

\bibitem[Kat89]{Kato}
K.~Kato,
  {\em Logarithmic structures of {F}ontaine-{I}llusie}, in:  
  {\em Algebraic analysis, geometry, and number theory} ({B}altimore,
  {MD}, 1988), pp.~191--224, Johns Hopkins Univ.\ Press, Baltimore, MD, 1989.

\bibitem[Liu02]{Liu}
Q.~Liu, 
  {\em Algebraic Geometry and Arithmetic Curves}, 
  Translated from French by R.~Ern\'e, Oxf.\ Grad.\ Texts Math.\ 6, Oxford Univ. Press, Oxford, 2002.  

\bibitem[LT13]{LiuTong}
Q.~Liu and J.~Tong, 
  {\em N\'eron models of algebraic curves}, 
  Trans.\ Amer.\ Math.\ Soc.\ \textbf{368} (2013), 7019--7043.

\bibitem[MW18]{MolchoWise}
S.~Molcho and J.~Wise, 
  {\em The logarithmic {Picard} group and its tropicalization}, preprint 
 \arXiv{1807.11364} (2018).

\bibitem[N\'er64]{neron_article}
A.~N\'eron, 
  {\em Mod\`eles minimaux des vari\'et\'es ab\'eliennes sur les corps locaux et
    globaux},
Publ.\ Math.\ Inst.\ Hautes \'Etudes Sci.\ \textbf{21} (1964). 

\bibitem[Ore18]{GiulioToricAdd}
G.~Orecchia, 
  {\em A criterion for existence of {N\'eron} models of {Jacobians}}, preprint 
  \arXiv{1806.05552} (2018).

\bibitem[Ore19]{GiulioMonodromyCriterion}
\bysame,  
  {\em A monodromy criterion for existence of {Neron} models of abelian
  schemes in characteristic zero}, preprint 
  \arXiv{1904.03886} (2019).

\bibitem[Ray70]{Raynaud}
M.~Raynaud, 
{\em Anneaux locaux hens\'eliens},
Lecture Notes in Math.\ 169, Springer-Verlag, Berlin-New York, 1970. 
  

\bibitem[Stacks]{stacks-project}
The Stacks Project Authors, 
  \emph{The Stacks Project},  
  \url{http://stacks.math.columbia.edu}.

\bibitem[Swe75]{Tensor_local}
M.\,E.~Sweedler, 
{\em When is the tensor product of algebras local?},
Proc.\ Amer.\ Math.\ Soc.\ \textbf{48} (1975). 

\end{thebibliography}
\end{document}